%% file: acceptedVersion.tex
\documentclass[11pt,a4paper,oneside]{amsart}
\include{top}
\usepackage{csquotes}

\newcommand{\nov}[3]{{\mathrm{Nov}({#1 #2, #3}})}
\DeclareMathOperator{\id}{id}
\newcommand{\cgr}[2]{#1 \llbracket #2 \rrbracket} 
\newcommand{\onto}{\twoheadrightarrow}

\def\Z{\mathbb{Z}}
\def\iff{if and only if }

\title[Profinite rigidity of fibring]{Profinite rigidity of fibring}

\usepackage[foot]{amsaddr}
\author{Sam Hughes}
\email{sam.hughes.maths@gmail.com}

\author{Dawid Kielak}
\email{kielak@maths.ox.ac.uk}

\address{Mathematical Institute, Andrew Wiles Building, University of Oxford, Oxford, OX2 6GG, UK}


\begin{document}

\maketitle

\begin{abstract}
We introduce the classes of TAP groups, in which various types of algebraic fibring are detected by the non-vanishing of twisted Alexander polynomials. We show that finite products of finitely presented LERF groups lie in the class $\TAP_1(R)$ for every integral domain $R$, and deduce that algebraic fibring is a profinite property for such groups. We offer stronger results for algebraic fibring of products of limit groups, as well as applications to profinite rigidity of Poincar\'e duality groups in dimension $3$ and RFRS groups.
\end{abstract}


\section{Introduction}
Let $G$ be a finitely generated residually finite group.  The \emph{profinite genus} of $G$, denoted by $\calg(G)$, is defined as the set of isomorphism classes of finitely generated residually finite groups with the same finite quotients as $G$. A group $G$ is called \emph{almost profinitely rigid} if $\calg(G)$ is finite and \emph{profinitely rigid} if $|\calg(G)|=1$.

The study of (almost) profinite rigidity has motivated and been the subject of a vast amount of research.  For example, finitely generated nilpotent groups are almost profinitely rigid \cite{Pickel1971} and so are polycyclic groups \cite{GrunewaldPickelSegal1980}.

Many groups are not profinitely rigid, for example, there are metabelian groups with infinite genus \cite{Pickel1974}, Platonov--Tavgen' showed that $F_2\times F_2$ is not profinitely rigid \cite{PlatonovTavgen1986}, Pyber \cite{Pyber2004} showed the genus could be uncountable,  Bridson--Grunewald gave examples of the failure of profinite rigidity amongst the class of finitely presented groups \cite{BridsonGrunewald2004} answering several questions of Grothendieck, and Bridson \cite{Bridson2016} showed that the profinite genus amongst finitely presented groups can be infinite.

In general, profinite rigidity remains a very mysterious subject.
Somewhat surprisingly, one family of groups for which  we are developing a fair deal of understanding of profinite phenomena is the family of fundamental groups of compact $3$-manifolds.
 One particularly noteworthy statement is the theorem of Bridson--McReynolds--Reid--Spitler~\cite{Bridsonetal2020} saying that the fundamental groups of some hyperbolic $3$-manifolds (including the Weeks manifold) are profinitely rigid, that is, each is distinguished from every other finitely generated residually finite group by the set of isomorphism classes of its finite quotients. In loc. cit. the authors conjecture that every Kleinian group is profinitely rigid.

Restricting attention solely to $3$-manifold groups, we have two remarkable results: First,  Jaikin-Zapirain~\cite{JaikinZapirain2020} showed that if the profinite completion of the fundamental group of a compact orientable aspherical $3$-manifold is isomorphic to that of $\pi_1(\Sigma)\rtimes \Z$ with $\Sigma$ a compact orientable surface, then the manifold fibres over the circle.  Second, Liu~\cite{Liu2020} proved that there are at most finitely many diffeomorphism types of finite-volume hyperbolic $3$-manifolds with isomorphic profinite completions of their fundamental groups.  The key point of Liu's work involves aligning a fibred map to $\Z$ from each pair of profinitely isomorphic fundamental groups and encoding the dynamics of them into the profinite completion. Both Jaikin-Zapirain's and Liu's theorems rely in a crucial way on the following result of Friedl--Vidussi proved in the sequence of papers \cite{FriedlVidussi2008,FriedlVidussi2011annals,FriedlVidussi2013}.

\begin{thm}[cf. {\cite[Theorem 1.1]{FriedlVidussi2011annals}}]
	\label{Friedl Vidussi general}
	Let $R$ be a  Noetherian unique factorisation domain (UFD). 
Let $M$ be a compact, orientable, connected $3$-manifold with empty or toroidal boundary. An epimorphism $\varphi \colon \pi_1(M) \to \Z$ is induced by a fibration $M \to \mathbb S^1$ \iff for every epimorphism $\alpha \colon \pi_1(M) \onto Q$ with finite image the associated first twisted Alexander polynomial $\Delta^{\varphi, \alpha}_{1,R}$ over $R$ is non-zero.
\end{thm}

The result relies in a key way on a special case proved by Friedl--Vidussi in an earlier work \cite{FriedlVidussi2008}, where the group $\pi_1(M)$ is additionally assumed to be locally extended residually finite (LERF, or subgroup separable). Once this is established, the above result follows by a series of arguments based on the work of Wilton--Zalesskii~\cite{WiltonZalesskii2010} and Wise~\cite{Wise2012}.

The interest in fibring has surpassed its roots in manifold topology finding numerous applications within the realm of geometric group theory, for example in the construction of subgroups of hyperbolic groups with exotic finiteness properties \cite{JankiewiczNorinWise2021,ItalianoMartelliMigliorini2020,ItalianoMartelliMigliorini2021,IsenrichMartelliPy2022,BattistaMartelli2022,Fisher2022,IsenrichPy2022}, exotic higher rank phenomena \cite{Kropholler2018a,Hughes2022}, the existence of uncountably many groups of type $\mathsf{FP}$ \cite{Leary2018,Leary2018b,KrophollerLearySoroko2020,BrownLeary2020}, a connection between fibring of RFRS groups and $\ell^2$-Betti numbers \cite{Kielak2020RFRS,Fisher2024,FisherHughesLeary2023}, and the construction of analogues of the Thurston polytope for various classes of groups \cite{FriedlLueck2017,FriedlTillmann2020,Kielak2020BNS,HenselKielak2021}.

Because of the widespread applicability of fibring, and in particular in view of the results of Friedl--Vidussi and Jaikin-Zapirain cited above, 
it is very desirable to be able to show that a group fibres if some group in its genus does. Moreover, since $3$-manifold groups are not inherently more interesting from the profinite perspective than other groups, it is entirely natural to try to find profinite properties of groups that ensure that fibring is shared by all groups in a genus.

Hence, the version of \cref{Friedl Vidussi general} for LERF groups is the starting point for our investigations. First, we introduce the notion of \emph{TAP groups} (standing for Twisted Alexander Polynomial), that is groups in which the twisted Alexander polynomials control algebraic fibring, see \cref{TAP def}.  Roughly, the twisted Alexander polynomials are invariants that describe the module structure of the homology of kernels of epimorphisms to  groups that are virtually $\ZZ$.
 
 We show that in fact all finitely presented LERF groups are TAP -- see \cref{LERF implies TAP} for the precise (more general) statement. This amounts to showing the following.

\begin{thmx}
	Let $G$ be a finitely presented LERF group and let $R$  be an integral domain. An epimorphism $\varphi \colon G \onto \Z$ is algebraically fibred \iff for every epimorphism $\alpha \colon G \onto Q$ with finite image the associated first twisted Alexander polynomial over $R$ is non-zero.
\end{thmx}
Here, an epimorphism $\varphi\colon G\onto \Z$ is \emph{algebraically fibred} if $\ker \varphi$ is finitely generated.  The group $G$ is \emph{algebraically fibred} if it admits such an epimorphism.  Also, we are talking about vanishing of Alexander polynomials over arbitrary integral domains, which might seem worrying, as the definition of the polynomial requires $R$ to be a Noetherian UFD. It does however make sense to talk about vanishing even when the polynomial is itself not well defined, see \cref{def vanishing}.

We use the above to show that for finite products of finitely presented LERF groups, algebraic fibring is a profinite property.  

\begin{thmx}
	\label{profinite intro}
	Let $G_A$ and $G_B$ be finite products of finitely presented LERF groups.  Suppose $G_A$ and $G_B$ have isomorphic profinite completions. Then, the group $G_A$ is algebraically fibred \iff $G_B$ is. 
\end{thmx} 
 Again, this is really a corollary of the more general \cref{dim 1} combined with \cref{cor.products,LERF symmetric BNS}.
 
This result can be used to study profinite properties of those high-dimen\-sion\-al manifolds whose fundamental groups are products of LERF groups; examples include products of surfaces, geometric $3$-manifolds, higher dimensional nil and sol manifolds, and many bundles where generic fibres have amenable fundamental groups.
This is significant progress in the study of such manifolds, since the tools that work in $3$-manifolds can be adapted to higher dimensions only in exceptional circumstances.
 
 Limit groups (and more generally residually free groups) are widely studied, see for example \cite{Sela2001I,Wilton2008,BestvinaFeighn2009} and the references therein.  An even more general (and more technical) result is given by \cref{prop.profinite.fibre}, where we deal with algebraic semi-fibring of higher degree (see \cref{def semi-fibred}).
 Combining this with work of Bridson, Howie, Miller and Short \cite{BridsonHowieMillerShort2009} on finiteness properties of residually free groups, we show the following.
 
 \begin{thmx}
 	\label{prod limit profinite intro}
 	Let $\FF$ be a finite  field.  Let $G_A$ and $G_B$ be profinitely isomorphic finite products of limit groups.  The group $G_A$ is $\mathsf{FP}_n(\FF)$-semi-fibred if and only if $G_B$ is.
 \end{thmx}
 
\cref{profinite intro} finds another application in the study of profinite rigidity of Poincar\'e duality groups which should be viewed as a step towards the `profinite' Cannon conjecture: \emph{If $G$ is a word hyperbolic group whose profinite completion is a profinite-Poincar\'e duality group in dimension $3$, then $G$ is the fundamental group of a closed connected hyperbolic $3$-manifold.}

\begin{thmx}
	\label{pd3 intro}
	Let $G_A$ be a LERF $\mathsf{PD}_3$-group. Let $G_B$ be the fundamental group of a closed connected hyperbolic $3$-manifold.
	If $\widehat{G_A}\cong\widehat{G_B}$, then $G_A$ is the   
	fundamental group of a closed connected hyperbolic $3$-manifold.
\end{thmx}

Finally, \cref{detecting betti} implies that for a cohomologically good RFRS group $G$ of type $\mathsf{F}$, the profinite completion of $G$ detects the degree of acyclicity of $G$ with coefficients in the skew-field $\cald_{\FF G}$ introduced by Jaikin-Zapirain; here $\FF$ is a finite  field. The skew-field $\cald_{\FF G}$ can be thought of as an analogue of the Linnell skew-field in positive characteristic, and hence can be used to define a positive-characteristic version  of $\ell^2$-homology.

\subsection*{Acknowledgements}
This work has received funding from the European Research Council (ERC) under the European Union's Horizon 2020 research and innovation programme (Grant agreement No. 850930).  The first author thanks Martin Bridson for a helpful conversation.  Both authors thank the referees for helpful comments which improved the exposition of this paper.

\section{Preliminaries}

Throughout, all rings are associative and unital, and ring morphisms preserve units. All modules are left-modules, unless stated otherwise. In particular, resolutions will be left-resolutions (that is consisting of left modules), and hence coefficients in homology will be right-modules (and quite often bimodules).

Integral domains and fields are always commutative.

\subsection{Bieri--Neumann--Strebel invariants}

Let $R$ be a ring, $G$ a group, and $\varphi \colon G \to \RR$ a non-trivial homomorphism. Observe that
\[
G_\varphi = \{ g \in G \mid \varphi(g) \geqslant 0\}
\]
is a monoid.

\begin{defn}[Homological finiteness properties]
	We say that a monoid $M$ is of \emph{type $\mathsf{FP}_n(R)$} if the trivial $M$-module $R$ admits a resolution $C_\bullet$ by projective $RM$-modules in which $C_i$ is finitely generated for all $i\leqslant n$.
\end{defn}

Since every group is a monoid, the definition readily applies to groups as well.

The definition above is standard; we will sporadically mention also other standard finiteness properties, like type $\mathsf{FP}(R)$ and $\mathsf{F}$. Note that $G$ is of type $\mathsf{FP}_1(R)$ \iff it is finitely generated, and if it is finitely presentable then it is of type $\mathsf{FP}_2(R)$ for every ring $R$.

\begin{defn}
	We say that $\varphi$ lies in the \emph{$n$th BNS invariant over $R$}, and write $\varphi \in \Sigma^n(G;R)$, if $G_\varphi$ is of type $\mathsf{FP}_n(R)$.  
	
	We set $\Sigma^\infty(G;R) = \bigcap_n \Sigma^n(G;R)$.  Here we are considering $\Sigma^n(G;R)$, for $n\in\NN\cup\{\infty\}$, as subsets of $H^1(G;\RR)\setminus \{0\}$.
\end{defn}
	
	The first BNS invariant $\Sigma^1(G;R) = \Sigma^1(G)$ is independent of $R$.  It was introduced by Bieri--Neumann--Strebel in \cite{BieriNeumannStrebel1987}. The higher (homological) invariants defined above were introduced by Bieri--Renz \cite{BieriRenz1988} for $R=\ZZ$. The definition for general $R$ appears for example in the work of Fisher \cite{Fisher2024}. Fisher's paper also contains the following straight-forward generalisation of the work of Bieri--Renz.
		
	\begin{thm}[{\cite[Theorem 6.5]{Fisher2024}}, {\cite[Theorem 5.1]{BieriRenz1988}}] Let $G$ be a group of type $\mathsf{FP}_n(R)$. Suppose that $\varphi \colon G \to \ZZ$ is a non-trivial homomorphism.
		The kernel $\ker \varphi$ is of type $\mathsf{FP}_n(R)$ if and only if $\{\varphi, -\varphi\} \subseteq \Sigma^n(G;R)$.
	\end{thm}

    We will often refer to a homomorphism $\varphi\colon G\to \RR$ as a \emph{character} and a homomorphism $\varphi\colon G\to \ZZ$ as an \emph{integral character}.
 
	\begin{defn}
		\label{def semi-fibred}
		A non-trivial character $\varphi \colon G \to \Z$ is \emph{$\mathsf{FP}_n(R)$-fibred} if $\ker \varphi$ is of type $\mathsf{FP}_n(R)$. An $\mathsf{FP}_1(R)$-fibred character will be also called \emph{algebraically fibred}; this last notion is independent of $R$.
		
		 Similarly, an integral character in $\Sigma^n(G;R) \cup -\Sigma^n(G;R)$ will be called \emph{$\mathsf{FP}_n(R)$-semi-fibred}, and a character in $\Sigma^1(G) \cup -\Sigma^1(G)$ will be called \emph{algebraically semi-fibred}.
		 
		 A group $G$ will be called \emph{algebraically fibred} if it admits an algebraically fibred character.
	\end{defn}

The terminology `semi-fibred' is new. It is meant to capture the idea that a character behaves like a fibred character, but its negative might not.

	 The invariant $\Sigma^1(G)$ admits a number of alternative definitions. Let us now discuss one of them.
	
	\begin{defn}
		Let $B$ be a group, let $A,C\leqslant B$, and suppose that there exists an isomorphism $\iota\colon A\to C$.  The \emph{HNN extension} $B\ast_{\iota}$ with \emph{base group} $B$ and \emph{associated subgroups} $A$ and $C$ is defined by 
		\[B\ast_{\iota} =  B\ast \langle t \rangle / \langle \! \langle \{ t^{-1}at=\iota(a) : a\in A \} \rangle \! \rangle. \]
		The HNN extension is \emph{ascending} if $C=B$ and \emph{descending} if $A=B$. If it is ascending but not descending, it is \emph{properly ascending}.
	\end{defn}
	
	\begin{prop}[{\cite{Brown1987}}]
		\label{HNN criterion}
		Let $G$ be a finitely generated group. An epimorphism $\varphi \colon G \to \Z$ lies in $\Sigma^1(G)$ if and only if there exists an isomorphism $\rho \colon G \to B\ast_\iota$ where $B$ is finitely generated, the HNN extension $B\ast_\iota$ is descending, and $\varphi$ is equal to the composition of $\rho$ with the quotient map $B\ast_\iota \to B\ast_\iota/\langle\!\langle B \rangle\!\rangle = \langle t \rangle = \Z$. 
	\end{prop}

An observant reader will notice that Brown's original statement uses ascending, rather than descending HNN extensions. This has to do with left/right conventions for modules used in the definition of $\Sigma^1(G)$.

\subsection{Twisted Alexander polynomials}
\label{section TAP}
The following definitions are taken from Friedl and Vidussi's survey \cite{FriedlVidussi2011survey}.  However, we have taken liberty to phrase them in terms of group homology as opposed to the homology of a topological space with twisted coefficients.

Let $R$ be an integral domain (we always assume these to be commutative) and let $R[t^{\pm 1}]$ the ring of Laurent polynomials over $R$ in an indeterminate $t$. 
 Let $\alpha\colon G\twoheadrightarrow Q$ be a finite quotient of $G$.  This induces an $RG$-bimodule structure on the free $R$-module $RQ$ induced by left and right multiplication precomposed with $\alpha$ -- another way to say it is that $RQ$ is a quotient ring of $RG$, and this way becomes an $RG$-bimodule.  Let $\varphi\in H^1(G;\ZZ)$ be a cohomology class considered as a homomorphism $\varphi\colon G\to \ZZ$.  Consider $RQ[t^{\pm 1}]$ equipped with the $RG$-bimodule structure given by 
\[g.x = t^{\varphi(g)}\alpha(g)x, \quad x.g = xt^{\varphi(g)}\alpha(g) \]
 for $g \in G, x \in  RQ[t^{\pm 1}]$.
Note that $RQ[t^{\pm 1}] = R(\Z \times Q)$, and the action is multiplication precomposed with $\varphi \times \alpha$, as above.

For $n\in\ZZ$, we define the \emph{$n$th twisted (homological) Alexander module of $\varphi$ and $\alpha$} to be $H_n(G;RQ[t^{\pm 1}])$, where $RQ[t^{\pm 1}]$ has the non-trivial module structure described above. Observe that $H_n(G;RQ[t^{\pm 1}])$ also has the structure of a left 
 $R[t^{\pm 1}]$-module.
We will  denote the module by $H^{\varphi,\alpha}_{n,R}$.  If $G$ is of type $\mathsf{FP}_n(R)$, then the $n$th twisted Alexander module is a finitely generated $R[t^{\pm 1}]$-module.  Moreover, it is zero whenever $n<0$ or $n$ is greater than the cohomological dimension of $G$ over $R$.

More generally, for two groups $Z$ and $Q$, given two group homomorphisms $\alpha \colon G \to Q$ and $\varphi \colon G \to Z$, we will sometimes use $H^{\varphi,\alpha}_{n,R}$ to denote $H_n(G;R(Z \times Q))$ with the $RG$-bimodule structure on $R(Z \times Q)$ being multiplication precomposed with $\varphi \times \alpha$.

For any integral domain $S$ and any finitely generated $S$-module $M$, define the \emph{rank} of $M$ to be $\rk_{S}M=\dim_{\Frac(S)}  \Frac(S)  \otimes_S  M$, where $\Frac(S)$ denotes the classical field of fractions (that is, the  Ore localisation) of $S$.  When $S$ is additionally a UFD, the \emph{order} of $M$ is the greatest common divisor of all maximal minors in a presentation matrix of $M$ with finitely many columns.  The order of $M$ is well defined up to a unit of $S$ and depends only on the isomorphism type of $M$. We do not require there to be only finitely many minors, nor $S$ to be Noetherian. In UFDs, any number of elements has a greatest common divisor.


\begin{defn}
Suppose that $G$ is of type $\mathsf{FP}_n(R)$, with $R$ being a  UFD.  Let $\varphi\colon G\to \Z$ be a homomorphism and $\alpha\colon G\onto Q$ be an epimorphism with $Q$ finite.  The \emph{$n$th twisted Alexander polynomial} $\Delta_{n,R}^{\varphi,\alpha}(t)$ over $R$ with respect to $\varphi$ and $\alpha$ is defined to be the order of the $n$th twisted (homological) Alexander module of $\varphi$ and $\alpha$, treated as a left $R[t^{\pm 1}]$-module. Note that $R[t^{\pm 1}]$ is a  UFD since $R$ is.
\end{defn}


\begin{example}
	Let us compute two instances of the twisted Alexander module for the Baumslag--Solitar group $G = \BS(1,2) = \langle a,t\ |\ tat^{-1}=a^2\rangle$. We take $R = \Z$. Let $\varphi \colon G \to \Z$ be the map killing $a$.
	
	In the first instance, let $Q$ be the trivial group, and so $\alpha$ is the trivial map. In this case the first twisted homological Alexander module of $\varphi$ and $\alpha$ is simply $H_1(G;\Z [t^{\pm 1}])$, which is the homology of the chain complex
	\[
	\Z [t^{\pm 1}] \to \Z [t^{\pm 1}] \oplus \Z [t^{\pm 1}] \to \Z [t^{\pm 1}]
	\]
	where the first map is the matrix $(1- 2t, 0)$ and the second is the transpose of $(0, 1-t)$. It is immediate that the homology is isomorphic as a $\Z [t^{\pm 1}]$-module to $\Z [t^{\pm 1}]/(1-2t)$, and hence the twisted Alexander polynomial is $1-2t$, which is also the untwisted Alexander polynomial, as $\alpha$ is trivial.
	
	Now let us consider a more interesting situation in which $Q = S_3$ is the permutation group of rank three, where $\alpha(a) = (1 2 3)$ and $\alpha(t) = (1 2)$. The twisted Alexander module in this case is the homology of 
		\[
	\Z Q [t^{\pm 1}] \to \Z Q [t^{\pm 1}] \oplus \Z Q [t^{\pm 1}] \to \Z Q [t^{\pm 1}]
	\]
	where the first map is the matrix $(1- (1+a)t, a-1)$ and the second is the transpose of $(1-a, 1-t)$. Since the first entry of the first matrix is of the form $1-xt$, with $x \in \Z Q$, it is easy to see that every element of the middle module can be written as $y+zt^n$, where $y$ lies in the image of the first map, $z = (z_1, z_2)$ has $z_1 \in \Z Q$, and $n \in \Z$. If we additionally require $y+zt^n$ to lie in the kernel of the second map, then we immediately see that $z_2 = 0$, since otherwise $z_2t^n(1-t)$ would not lie in $\Z Q t^n$, but $z_1t^n(a-1)$ would, and the two elements have to cancel. But then $z_1 \in \Z Q$ must have equal coefficients on all elements lying in the same $\langle a \rangle$-cosets in $S_3$.
	This tells us that the kernel of the second map is equal to the (non-direct) sum of the image of the first map, and a free $\Z [t^{\pm 1}]$-module of rank two spanned by $(1+a+a^2,0)$ and $((1+a+a^2)t,0)$. Taking the quotient of this latter module by its intersection with the image of the first map yields  $\Z [t^{\pm 1}]/(1-2t) \oplus \Z [t^{\pm 1}]/(1-2t)$, and this is precisely the twisted Alexander module.  The twisted Alexander polynomial is equal to $(1-2t)^2$.
\end{example}

Since we will be concerned with the vanishing of $\Delta_{n,R}^{\varphi,\alpha}(t)$, let us record a number of equivalent statements. From now on we drop the requirement on $R$ being a UFD.

\begin{lem}
		\label{vanishing}
	Let $R$ be an integral domain, and let $F = \Frac(R)$. Suppose that $G$ is of type $\mathsf{FP}_n(R)$.
	The following are equivalent:
	\begin{enumerate}
		\item \label{vanishing def 3} $\rk_{R[t^{\pm1}]} H^{\varphi,\alpha}_{n,R} = 0$;
		\item \label{vanishing def 2} $H^{\varphi,\alpha}_{n,R}$ is a torsion $R[t^{\pm 1}]$-module.
		\item \label{vanishing def 1} $H^{\varphi,\alpha}_{n,F}$ is a torsion $F[t^{\pm 1}]$-module.
		\item \label{vanishing def 4} $H^{\varphi,\alpha}_{n,F}$ is a finitely generated $F$-module.
	\end{enumerate}
	If additionally $R$ is a UFD, then these are equivalent to
	\begin{enumerate}[resume]
		\item \label{vanishing def 5} $\Delta_{n,R}^{\varphi,\alpha}(t) \neq 0$.
	\end{enumerate}
\end{lem}
\begin{proof}[Sketch proof]
	We offer only a sketch, since these equivalences are standard.
	
	Items \eqref{vanishing def 2} and \eqref{vanishing def 1} are equivalent since $F$ is a flat $R$-module. Items \eqref{vanishing def 1}, \eqref{vanishing def 4}, and \eqref{vanishing def 3} are equivalent thanks to the classification theorem for finitely generated modules over a PID, since $F[t^{\pm 1}]$ is a PID; one also needs to note that $\Frac(R[t^{\pm 1}]) = \Frac(F[t^{\pm 1}])$.
	
	The equivalence of \eqref{vanishing def 5} with the other ones is explained in {\cite[Remark 4.5, Clause 2]{Turaev2001}}.
\end{proof}

\begin{dfn}
	\label{def vanishing}
	Let $R$ be an integral domain, $\varphi \colon G \to \Z$ be a homomorphism, and $\alpha \colon G \onto Q$ be a finite quotient. We say that $\varphi$ has \emph{non-vanishing $n$th Alexander polynomial  twisted by $\alpha$} if $\rk_{R[t^{\pm1}]} H^{\varphi,\alpha}_{n,R} = 0$. If this holds for $\alpha = \tr \colon G \to \{1\}$, we say that the \emph{$n$th Alexander polynomial} does not vanish; if the statement holds for all choices of $\alpha$, we say that  $\varphi$ has \emph{non-vanishing $n$th twisted  Alexander polynomials}.
\end{dfn}

It may seem strange to define non-vanishing of an object in terms of vanishing of a different object, but indeed, by \Cref{vanishing}, if $R$ is a  UFD, then $\rk_{R[t^{\pm1}]} H^{\varphi,\alpha}_{n,R} = 0$ is equivalent to $\Delta_{n,R}^{\varphi,\alpha}(t)\neq 0$.  \cref{vanishing} also shows that in \Cref{def vanishing} we may replace $R$ by $\Frac(R)$.

\begin{lem}
	\label{untwisting}
	The $n$th Alexander polynomial of $\varphi$ twisted by $\alpha$ vanishes
	\iff 
	the $n$th (untwisted) Alexander polynomial of $\varphi\vert_{\ker \alpha}$ vanishes. Moreover, if $R$ is a UFD then the corresponding twisted Alexander polynomials are equal.
\end{lem}
\begin{proof}
This was proved in \cite[Lemma~3.3]{FriedlVidussi2008a}, we include a proof for completeness.	
 We need to compare the  $R[t^{\pm 1}]$-modules 
	\[H_n(G;RQ[t^{\pm 1}]) \textrm{ and } H_n(\ker \alpha;R[t^{\pm 1}]).\]
	 Shapiro's lemma shows that these modules are isomorphic, since  $RQ[t^{\pm 1}]$ is isomorphic to the  induced right $RG$-module of the right $R(\ker \alpha)$-module $R[t^{\pm 1}]$.
\end{proof}

The following result is well known for $3$-manifolds and has appeared in several places \cite{KirkLivingston1999,ClayRolfson2012,GodaKitanoMorifuji2005,FriedlKim2006}; in fact, it appears to date back to work of Milnor \cite{Milnor1967}.  We include a proof in the group theoretic setting for completeness.
\begin{prop}\label{prop.easy}
Let $R$ be an
  integral domain.  Let $G$ be a group of type $\mathsf{FP}_n(R)$ and let $\varphi\colon G\to \ZZ$ be a non-trivial character.  If $\varphi$ is $\mathsf{FP}_n(R)$-fibred, then its $k$th twisted Alexander polynomials never vanish for $k\leqslant n$.
\end{prop}

\begin{proof}
Since $\varphi$ is $\mathsf{FP}_n(R)$-fibred, $G$ splits as a semi-direct product $A\rtimes\ZZ$ with $A = \ker \varphi$ of type $\mathsf{FP}_n(R)$. Let $m = |\Z : \im \varphi|<\infty$.
  Now, let $\alpha\colon G\twoheadrightarrow Q$ be an epimorphism of $G$ onto a finite group and let $RQ[t^{\pm 1}]$ be the right $RG$-module with action given by $\varphi$ and $\alpha$. Applying \cite[III.6.2 and III.8.2]{Brown1982} yields that 
  \[
  H_\bullet(G;RQ[t^{\pm 1}]) = H_\bullet(G;\bigoplus_m RQ[t^{\pm m}]) \cong H_\bullet(A;\bigoplus_m RQ)
  \]
   as $R$-modules.  Now, since $A$ is of type $\mathsf{FP}_n(R)$ and $Q$ is finite it follows that  for $k\leqslant n$ the $R$-module $H_\bullet(A;\bigoplus_m RQ)$ is finitely generated. Such a module cannot contain a copy of $R[t^{\pm 1}]$, and  therefore $H_k(G;RQ[t^{\pm 1}])$ is a torsion  $R[t^{\pm 1}]$-module. We are done by \cref{vanishing}.
\end{proof}

\begin{prop}
	\label{semifibred so Alex}
Let $R$ be an integral domain.  Let $G$ be a group of type $\mathsf{FP}_n(R)$, and let $\varphi\colon G\to \ZZ$ be an $\mathsf{FP}_n(R)$-semi-fibred character.
	  The $k$th twisted Alexander polynomials of ${\varphi}$ are non-zero for all $k\leqslant n$.
\end{prop}

Note that the case $n=1$ is implied by \cite[Lemma~4.1]{FriedlVidussi2016}; our method of proof is quite distinct.

\begin{proof}
	Since $G$ is of type $\mathsf{FP}_n(R)$, we find a projective resolution $C_\bullet$ of the trivial $G$-module $R$ with $C_k$ a finitely generated $RG$-module for every $k \leqslant n$. We replace $\varphi$ by $-\varphi$ if needed, and assume that $\varphi \in \Sigma^n(G;R)$; note that this replacement does not affect the vanishing of twisted Alexander polynomials.

    The \emph{Novikov ring} $\nov R G \varphi$ is the ring of twisted Laurent power series with coefficients in $R(\ker \varphi)$ and with variable $t \in G$ with $\varphi(t)=1$, where the twisting is given by the conjugation action of $t$ on $\ker \varphi$; multiplication in $\nov R G \varphi$ induces a right $RG$-module structure on $\nov R G \varphi$.
	
	Using Fisher's version of Sikorav's theorem \cite[Theorem 5.3]{Fisher2024}, we find a partial chain contraction for $C_\bullet$ over the Novikov ring  $\nov R G \varphi$ in the following sense: Denote the differentials of $C_\bullet $ by $\partial_i \colon C_i \to C_{i-1}$. We find $\nov R G \varphi$-module morphisms \[A_i \colon \nov R G \varphi \otimes_{R G}C_i \to \nov R G \varphi \otimes_{R G} C_{i+1}\]
	 such that for every $i \leqslant n$ we have $A_{i-1} \partial'_{i} + \partial'_{i+1} A_{i} = \id$ where $\partial'_i = \id_{\nov R G \varphi} \otimes_{R G} \partial_i$, and $A_{-1} = 0, \partial'_{-1}=0$. 
	
	Now, let $\alpha \colon G \onto Q$ be an epimorphism with $Q$ finite. Dividing $G$ by the normal subgroup $K = \ker \alpha \cap \ker \varphi$ induces a ring morphism
	\[
	\beta \colon \nov R G \varphi \to \nov R {(G/K)} {\psi} 
	\]
	where $\psi \colon G/K \to \mathbb Z$ is induced by $\varphi$. 
	
	To compute the homology $H_i(G;\nov R {(G/K)} {\psi} )$, we need to tensor the chain complex $C_\bullet$ with $\nov R {(G/K)} {\psi}$. This has the same effect as applying the homomorphism $\beta$ to the free $\nov R {G} {\varphi}$-modules and differentials constituting the complex  $\nov R G \varphi \otimes_{R G}C_\bullet$. This implies that 	
	applying $\beta$ to the entries of the matrices $A_i$ gives us another set of chain contractions with the same properties as above, and hence
	\[
	H_i(G;\nov R {(G/K)} {\psi} ) = 0
	\]
	for all $i \leqslant n$.
	
	The ring $\nov R {(G/K)} {\psi}$ is isomorphic to $\bigoplus_Q \nov R {(\ker \alpha /K)} {\psi}$ as an $R (\ker \alpha /K)$-module, and hence also as an $R (\ker \alpha)$-module, and so
	\[
	H_i\big( \ker \alpha ; \nov R {(\ker \alpha /K)} {\psi} \big) = 0
	\]  
	for all $i \leqslant n$.
	Arguing with chain contractions as before, we see that
		\[
	H_i\big(\ker \alpha ; \nov {\Frac(R)} {(\ker \alpha /K)} {\psi}\big) = 0
	\]  
	for all $i \leqslant n$.
	
	 Now, $\ker \alpha /K \cong \mathbb Z$, and therefore $\nov {\Frac(R)} {(\ker \alpha /K)} {\psi})$ is the field of Laurent power series in a single variable $t$ and coefficients in $\Frac(R)$, where $t \in \ker \alpha$ is mapped by $\psi$ to a generator of $\mathbb Z$. This field contains the field  $R(t)$ of rational functions in a single variable and coefficients in $R$ in the obvious way. Since $R(t)$ is a right $R  (\ker \alpha)$-submodule of  $\nov {\Frac(R)} {(\ker \alpha /K)} {\psi})$, and 	since  $\nov {\Frac(R)} {(\ker \alpha /K)} {\psi})$ is a flat $R(t)$-module as both are skew-fields, we conclude that
		\[
	0 = H_i(\ker \alpha ; R(t)).
	\]  
Now, using flatness of localisations, we obtain
	\[
	 H_i(\ker \alpha ; R(t))  = H_i(\ker \alpha ; R[t^{\pm 1}]) \otimes_{R[t^{\pm 1}]} R(t)
	 \]
	and therefore $H_i(\ker \alpha ; R[t^{\pm 1}])$ is a torsion $R[t^{\pm 1}]$-module. We are now done thanks to \cref{untwisting,vanishing}.
\end{proof}

\begin{example}
	\label{bs example}
    The Baumslag--Solitar group \[\BS(1,n) = \langle a,t\ |\ tat^{-1}=a^n\rangle\]
     has $H^1(\BS(1,n);\RR)\cong \RR$ with basis given by the character 
     \[\varphi\colon \BS(1,n)\twoheadrightarrow\langle t\rangle\cong\ZZ\]
      killing $a$.  The BNS invariant $\Sigma^1(\BS(1,n))$ consists only of the ray $\{ \lambda \varphi \mid \lambda \in (0,\infty)\}$.  It follows that for every  integral domain  $R$  and every finite quotient $\alpha\colon\BS(1,n)\twoheadrightarrow Q$, the twisted Alexander polynomials never vanish. (In fact, the polynomials can be computed by hand rather easily.) Note that $\BS(1,n)$ splits as $\ZZ[1/n]\rtimes\ZZ$, where $\ZZ$ acts as multiplication by $n$, so $\ker(\varphi)$ is not finitely generated.
\end{example}

\section{TAP groups}

\subsection{The definition}

\begin{defn}
	\label{TAP def}
	Let $R$ be an integral domain.  We say that a group $G$ of type $\mathsf{FP}_n(R)$ is in the class $\TAP_n(R)$ if for every non-trivial character $\varphi\in H^1(G;\ZZ)$ the following property holds:
	\begin{enumerate}
		
		\item[] \begin{tabular}{p{0.85\textwidth}}\emph{$\varphi$ is $\mathsf{FP}_n(R)$-semi-fibred if and only if for each $i \leqslant n$ its twisted $i$th Alexander polynomials never vanish.}\end{tabular}
		
	\end{enumerate}
	We allow $n= \infty$ in the above definition.
	
	The definition is best motivated and explained by the following slogan: \emph{``A group is in $\TAP_n(R)$ if and only if twisted Alexander polynomials detect algebraic semi-fibring over $R$ up to dimension $n$''.}
\end{defn}

Note that in view of \cref{bs example} it is more natural to use semi-fibring rather than fibring in the definition above.  Indeed, vanishing of twisted Alexander polynomials alone cannot distinguish between fibring and semi-fibring.

\begin{example}
	\cref{Friedl Vidussi general} by Friedl--Vidussi shows that fundamental groups of compact, orientable, connected $3$-manifolds with empty or toroidal boundary are in $\TAP_1(R)$.  In fact they are in $\TAP_\infty(R)$: Indeed, the first BNS invariant of a compact $3$-manifold group $G$ is symmetric \cite[Corollary F]{BieriNeumannStrebel1987}, that is $\Sigma^1(G)=-\Sigma^1(G)$,  and a finitely generated subgroup of a $3$-manifold group is of type $\mathsf{F}_\infty$ by Scott's compact core theorem \cite{Scott1973a}.  Hence, such a subgroup is type $\mathsf{FP}_\infty(R)$ over every $R$. 
\end{example}	

\begin{example}
	A non-example is given by $G=S\wr \ZZ$ where $S$ is an 
	infinite simple group.  Note that such a group has an obvious map $\varphi\colon G\twoheadrightarrow\ZZ$ and this map is a basis for $H^1(G;\RR)\cong \RR$. The group $G$ admits an automorphism that acts on $H^1(G;\RR)$ as minus the identity, and hence the BNS invariants of $G$ must be symmetric.	
	Therefore $\Sigma^1(G;R)$ is empty since $\ker\varphi=\bigoplus_\ZZ S$ is not finitely generated.
	Now, every finite quotient of $G$ is cyclic and the corresponding kernel is isomorphic to $S^n \wr \ZZ$ for some $n$; the Alexander polynomial of such a group is equal to $1$, since the relevant $R$-module is $H_1(S^n \wr \ZZ; R[t^{\pm 1}]) \cong H_1(\bigoplus_{\Z} S^n; R) = 0$. This shows that $G$ is not in $\TAP_1(R)$ for any $R$.
\end{example}

\begin{example}
	Another non-example is provided by every group that admits a character that is $\mathsf{FP}_2(\QQ)$-semi-fibred without being $\mathsf{FP}_2(\Z)$-semi-fibred. Such a group cannot be in $\TAP_2(\Z)$, since if it were then the character would have non-vanishing twisted second Alexander polynomials over $\QQ$ by \cref{semifibred so Alex}, and hence over $\Z$ by \cref{vanishing}, and then $\TAP_2(\Z)$ would show that the character is $\mathsf{FP}_2(\Z)$-semi-fibred. An explicit example of a group satisfying the requirement is every RAAG based on a triangulation of the real projective plane; the character will then be the Bestvina--Brady character \cite{BestvinaBrady1997}.
\end{example}

We will be primarily interested in profinite aspects of TAP groups, but the property has also other uses.

Italiano--Martelli--Migliorini in \cite{ItalianoMartelliMigliorini2020} introduced a finite-volume hyperbolic $7$-manifold whose fundamental group maps onto $\Z$ with finitely presented kernel. Fisher \cite{Fisher2022} showed that by passing to a suitable finite cover, one obtains a finite-volume hyperbolic $7$-manifold $M$ with $G = \pi_1(M)$ and an epimorphism $\varphi \colon G \to \Z$ with kernel that is finitely presented and of type $\mathsf{FP}(\QQ)$.

Twisted Alexander polynomials could potentially be used to show that the $7$-manifold fibres over the circle:  Suppose that $G$ lies in $\TAP_7(\Z)$ and that
\[\Sigma^7(G;\Z) = -\Sigma^7(G;\Z).\]
 Since $\varphi \in \Sigma^7(\pi_1(M);\QQ)$, we see that the twisted Alexander polynomials of $M$ over $\QQ$ never vanish in dimensions $1$ to $7$. This means that the polynomials over $\Z$ never vanish either, and since $G$ is in $\TAP_7(\Z)$ we conclude that $\varphi \in \Sigma^7(\pi_1(M);\Z)$. Since the BNS invariant is also assumed to be symmetric, $\ker \varphi$ is finitely presented and of type $\mathsf{FP}_7(\Z)$, and hence is of type $\mathsf F$. If one now had a version of Farrell's theorem \cite{Farrell1972} for manifolds with boundary, one could conclude that $M$ fibres over the circle.

\subsection{Almost finitely presented LERF groups are \texorpdfstring{$\mathsf{TAP}_1(R)$}{TAP\textoneinferior (R)}}

Now that we have defined TAP, let us introduce the class of groups whose TAPness we want to establish.

\begin{defn}
	Let $G$ be a group. A subgroup $A \leqslant G$ is \emph{separable} if for every $g \in G \smallsetminus A$ there exists an epimorphism $\alpha \colon G \twoheadrightarrow Q$ with $Q$ finite such that
	\[
	\alpha(g) \not\in \alpha(A).
	\]
	
	A group $G$ is \emph{LERF} (or \emph{locally extended residually finite}, or \emph{subgroup separable}) if every finitely generated subgroup is separable.
\end{defn}

We will need some standard terminology related to graph-of-groups decompositions.

\begin{dfn}
	We say that a group $G$ \emph{splits over a subgroup $A$} if $G$ decomposes as a reduced graph of groups with a single edge and edge group $A$. Recall that a graph of groups is \emph{reduced} if every edge both of whose attaching maps are isomorphisms is a loop.
\end{dfn}

We are ready to state our main technical tool. The HHN extension case is a variation on the proofs from \cite{FriedlVidussi2008}.

\begin{prop}
	\label{graph of groups friedl vidussi}
	Let $G$ be a finitely generated group that splits over a separable subgroup. Let $\varphi \colon G \to \Z$ be a non-zero character that vanishes on the edge group. If for some  integral domain  $R$ the first twisted Alexander polynomials never vanish, then the splitting has only one vertex and $\varphi$ is algebraically fibred with kernel equal to the edge group.
\end{prop}
\begin{proof} 
		We need to consider two cases, depending on whether the splitting is an HNN extension or an amalgamated free product.
	
		Suppose first that $G$ splits as an HNN extension (in fact, this is always the case; we will prove this by contradiction later on). If both edge maps are isomorphisms, then the edge group is a normal subgroup, and quotienting by it yields $\Z$.    Hence, $\varphi$ is algebraically fibred with kernel equal to the edge group, as claimed. 
		Suppose now that at least one of the attaching maps is not a surjection. Let $A$ denote the image of this map, and let $B$ denote the vertex group.
	
	Let $\alpha \colon G \onto Q$ be an epimorphism with finite image.
	 Consider the Mayer--Vietoris sequence for an HNN-extension (see for instance \cite[Chapter~VII.9]{Brown1982}) with non-trivial coefficients $R Q[t^{\pm 1}]$ as in \cref{section TAP}, where the action of $A$ and $B$ on the module is inherited from $G$.  The sequence takes the following form:
	\begin{equation*}\label{eqn.MayerVietorisSequenceHNN} \begin{tikzcd}
	{}  & {\cdots} \arrow[r]                        & {H_1(G; R Q[t^{\pm 1}])} \arrow[lld] &   \\
	{H_0(A;RQ[t^{\pm 1}])} \arrow[r] & {H_0(B;R Q[t^{\pm 1}])} \arrow[r] & {H_0(G; R Q[t^{\pm 1}])} \arrow[lld] \\
	0
	\end{tikzcd}\end{equation*}
	Since $A \leqslant \ker \varphi$, we have a right $A$-module isomorphism 
	\[
	R Q[t^{\pm 1}] =  R[t^{\pm 1}] \otimes_R R Q
	\]
	where the action of $g \in A$ on $R[t^{\pm 1}] \otimes_R R Q$ is the diagonal action given by right-multiplication by $\alpha(g)$ on $R Q$ and the trivial action of $R[t^{\pm 1}]$.
	We also have an $R$-module isomorphism 
	\[H_0(A; R[t^{\pm 1}] \otimes_R R Q)\cong R[t^{\pm 1}] \otimes_R (R Q)_A\]
	 by the definition of zeroth homology, where $(-)_A$ denotes $A$-coinvariants.  
	
 By assumption, $H_1(G;R Q[t^{\pm 1}])$ is $R[t^{\pm 1}]$-torsion and it is clear that 
 \[H_0(G;R Q[t^{\pm 1}])\]
  is $R[t^{\pm 1}]$-torsion (see for instance \cite[Lemma~4.4]{FriedlVidussi2008}). Applying these observations in the trivial case $\alpha = \tr\colon G\to \{1\}, Q = \{1\}$, we see that $H_0(B;R[t^{\pm 1}])$ must contain a copy of 
 $R[t^{\pm 1}] \otimes_R R_A = R[t^{\pm 1}]$.
  If $\varphi\vert_B \neq 0$, then it is immediate that $H_0(B;R[t^{\pm 1}]) = (R[t^{\pm 1}])_B$ is a torsion  $R[t^{\pm 1}]$-module, yielding a contradiction. We conclude that $\varphi\vert_B = 0$, and hence we have
    $H_0(B;R Q[t^{\pm 1}])\cong R[t^{\pm 1}] \otimes_R (R Q)_B$ for all $\alpha$ and $Q$.
	
	Using the fact that $A$ is separable, we produce an 
	epimorphism $\alpha \colon G \twoheadrightarrow Q$ with finite image such that $\alpha(A)$ is a proper subgroup of $\alpha(B)$.
		Let $F=\Frac(R)$. 
	Note that $F(t)$, the field of rational functions, is a flat $R[t^{\pm1}]$-module.
	Tensoring the Mayer--Vietoris sequence above (with this choice of $\alpha$) 
	with $F(t)$ over $R[t^{\pm 1}]$ we see that
	\[
	\dim_{F(t)}  F(t) \otimes_R (R Q)_A=\dim_{F(t)}  F(t) \otimes_R (R Q)_B.
	\]
	Observe that $(R Q)_A$ is a free right $R$-module of rank $\vert Q : \alpha(A)\vert$, and similarly for $(R Q)_B$. The dimensions above pick up exactly the $R$-rank, and so we may conclude that
	\[\vert Q : \alpha(A)\vert=\vert Q : \alpha(B)\vert,\]
	contradicting $|\alpha(A)|<|\alpha(B)|$. 
	
	\smallskip

	Suppose now $G$ splits as an amalgamated free product, we will show now this is not the case, again by contradiction. Now, the edge group $A$ must be a proper subgroup of the vertex groups $B$ and $B'$, since otherwise the graph of groups would not be reduced. 
	
	We now consider the Mayer--Vietoris sequence for a free product with amalgamation:
	\begin{equation*}
	 \begin{tikzcd}[column sep=0.35cm]
	{}  & {\cdots} \arrow[r]                        & {H_1(G; R Q[t^{\pm 1}])} \arrow[lld] &   \\
	{H_0(A;RQ[t^{\pm 1}])} \arrow[r] & {H_0(B;R Q[t^{\pm 1}])\oplus H_0(B';R Q[t^{\pm 1}])} \arrow[r] & {H_0(G; R Q[t^{\pm 1}])} \arrow[lld] \\
	0
	\end{tikzcd}
	\end{equation*}
	Arguing as before with $\alpha=\tr$, we first see that $\varphi$ must vanish on precisely one of the vertex groups, say $B$ -- it cannot vanish on both since $\varphi \neq 0$. As before, we produce $\alpha \colon G \twoheadrightarrow Q$ such that $\alpha(A)<\alpha(B)$. After tensoring with $F(t)$ over $R[t^{\pm 1}]$ we obtain an isomorphism between $F(t) \otimes_{R[t^{\pm1}]} H_0(A;RQ[t^{\pm 1}])$ and 
	\[
	\big( H_0(B;R Q[t^{\pm 1}]) \otimes_{R[t^{\pm1}]} F(t)\big) \oplus \big( H_0(B';R Q[t^{\pm 1}]) \otimes_{R[t^{\pm1}]} F(t) \big).
	\]
	Since $\varphi\vert_{B'}$ is non-trivial, the  $R[t^{\pm1}]$-module $H_0(B';R Q[t^{\pm 1}])$ is torsion as before, and hence 
	\[ F(t) \otimes_{R[t^{\pm1}]}H_0(B';R Q[t^{\pm 1}]) =0.\]
Using dimensions over $F(t)$ we conclude that
	 $|\alpha(A)|=|\alpha(B)|$, as before. This is a contradiction. 
\end{proof}

We are now ready for our first main theorem.

\begin{thm}\label{prop.NonZeroAlexImpliesFibring}\label{LERF implies TAP}
 If $G$ is a LERF group of type $\mathsf{FP}_2(S)$ for some commutative ring $S$, then $G$ is in $\TAP_1(R)$ for every integral domain $R$.
\end{thm}

\begin{proof}
Fix an arbitrary integral domain $R$.  Let $\varphi\colon G\to \ZZ$ be a non-trivial character.  We aim to show that $\varphi$ is algebraically fibred if and only if for every epimorphism onto a finite group $\alpha\colon G\twoheadrightarrow Q$ the corresponding twisted Alexander polynomial does not vanish.  The `only if' direction is given by \cref{semifibred so Alex}.  For the other direction suppose that the twisted Alexander polynomials of $\varphi$ are non-zero.

Since $G$ is of type $\mathsf{FP}_2(S)$, by \cite[Theorem~A]{BieriStrebel1978} there exist finitely generated subgroups $A,B,C\leqslant G$ with $A,C \leqslant B$, and an isomorphism $\iota \colon A \to C$, such that $G$ splits as an HNN-extension $B\ast_{\iota}$, and dividing by $B$ coincides with $\varphi$.

Since $A$ is finitely generated and $G$ is LERF, we see that $A$ is separable. The result now follows from \cref{graph of groups friedl vidussi}.
\end{proof}

\begin{rmk}
	\label{LERF symmetric BNS}
	The proof of the above result together with \cref{semifibred so Alex} show that $\Sigma^1(G) = -\Sigma^1(G)$. This is a well-known fact that can be proved directly using \cref{HNN criterion}.
\end{rmk}

\cref{graph of groups friedl vidussi} can also be used in the setting of graphs of groups.

\begin{thm}
	Let $R$ be an integral domain.  Let $G$ be a finitely generated fundamental group of a finite reduced graph of groups $\calg$.  Let $\varphi\in H^1(G;\ZZ)$ be a non-zero character and suppose that $G$ is LERF.  If 
	the first twisted Alexander polynomials of $\varphi$  never vanish, then for every finitely generated edge group $A$ precisely one of the following holds:
	\begin{enumerate}
		\item either $G=A\rtimes\ZZ$ with $\varphi$ being the projection map,
		\item or $\varphi|_A\neq 0$.
	\end{enumerate}
\end{thm}
\begin{proof}
	Let $A$ be an arbitrary finitely generated edge group in the graph of groups decomposition and let $e$ be any edge with edge group $A$.
	The proof splits into two cases.
	
	If $e$ is non-separating, then we may collapse all the other edges and obtain a splitting of $G$ as an HNN extension with edge group $A$. Now, \cref{graph of groups friedl vidussi} tells us that if $\varphi\vert_A =0$, then $\varphi$ is algebraically fibred with kernel $A$, that is, $G = A \rtimes \Z$.
	
	If $e$ is a separating edge, then $G$ splits as a free product amalgamated over $A$. \cref{graph of groups friedl vidussi} tells us that $\varphi|_A\neq 0$.
\end{proof}

\subsection{Products of \texorpdfstring{$\TAP_1(R)$}{TAP\textoneinferior(R)} groups}

We will now discuss the structure of the BNS invariants for products of groups. When working over fields, this structure is completely understood in terms of BNS invariants of factors; over general commutative rings all we have is an inequality. To understand the inequality, recall that we have defined the BNS invariants $\Sigma^n(G;R)$ as subsets of $H^1(G;\RR) \smallsetminus \{0\}$. 
For a subset $U \subseteq H^1(G;\RR)$ we denote the complement by $U^c = H^1(G;\RR) \smallsetminus U$. In particular, we have $ \Sigma^0(G;R)^c = \{0\}$. 

When $G = G_1 \times G_2$, we have $H^1(G;\RR) = H^1(G_1;\RR) \oplus H^1(G_2;\RR)$. Given subsets $U_i \subseteq H^1(G_i;\RR)$ we define their \emph{join} to be
\[
U_1 \ast U_2 = \{tu_1 + (1-t)u_2 \mid u_i \in U_i, t \in [0,1] \}.
\] 

The following inequality is due to Meinert; see \cite{BieriGeoghegan2010} for the history of this and \cite{Gerhke1998} for a proof.  The ``moreover'' is due to Bieri--Geoghegan \cite{BieriGeoghegan2010} and for $R=\ZZ$ the inequality can be strict \cite{Schutz2008}.

\begin{thm}[Meinert's inequality]
Let $G_1$ and $G_2$ be groups of type $\mathsf{FP}_n(R)$ where $R$ is a commutative ring, and let $G=G_1\times G_2$.  Then \[\Sigma^n(G;R)^c\subseteq\bigcup^n_{p=0}\Sigma^p(G_1;R)^c\ast\Sigma^{n-p}(G_2;R)^c.\]
Moreover, equality holds if $R$ is a field.
\end{thm}

\begin{prop}\label{cor.products}
Let $R$ be an integral domain  and let $G_1$ and $G_2$ be finitely generated groups. 
  If $G_i$ is in $\TAP_1(R)$ for $i=1,2$, then $G_1\times G_2$ is in $\TAP_1(R)$.
\end{prop}
\begin{proof}
Let $G=G_1\times G_2$.  
Suppose that there exists $\varphi\colon G\twoheadrightarrow\ZZ$ that is not algebraically semi-fibred and is non-zero.  It suffices to show that there exists a finite quotient $\alpha\colon G\twoheadrightarrow Q$ such that the corresponding twisted Alexander polynomial vanishes.

By Meinert's inequality, we have
\[\varphi\in (\Sigma^1(G_1;R)^c\ast \{0\}) \cup (\{0\} \ast\Sigma^1(G_2;R)^c). \]
In particular, for exactly one $i\in\{1,2\}$ we have $\varphi|_{G_i}=0$.  Suppose without loss of generality that $i=2$.  

Now, we have a splitting $\ker(\varphi)=\ker(\varphi|_{G_1})\times G_2$.  Since $G_1$ lies in $\TAP_1(R)$, there exists a finite quotient $\alpha_1\colon G_1\twoheadrightarrow Q$ such that the module $H_{1,R}^{\varphi|_{G_1},\alpha_1}$ is not  $R[t^{\pm 1}]$-torsion, and hence contains a free $R[t^{\pm 1}]$-module. Let $F$ denote $\Frac(R)$. Since $F$ is a flat $R$-module, and since $\dim_F F \otimes_R R[t^{\pm 1}] = \infty$,  we immediately see that
\[
\dim_F F \otimes_R H_{1,R}^{\varphi|_{G_1},\alpha_1} = \infty.
\] 

Define $\alpha\colon G\twoheadrightarrow Q$ to be the composite $G\twoheadrightarrow G_1\twoheadrightarrow Q$.   Applying Shapiro's lemma (as in the proof of \cref{untwisting}), and then \cite[III.8.2]{Brown1982}, gives isomorphisms of $R$-modules  
\[H_{1,R}^{\varphi,\alpha}\cong H_{1,R}^{\varphi|_{\ker\alpha},\tr}\cong H_1(\ker(\varphi)\cap\ker(\alpha);R), \]
but $\ker(\varphi)\cap\ker(\alpha)\cong (\ker(\varphi|_{G_1})\cap\ker(\alpha_1))\times G_2$. 
  It follows that we can compute $H_{1,R}^{\varphi,\alpha}$ by the K\"unneth spectral sequence (note that $R$ is not necessarily a PID so we cannot use the K\"unneth formula, see \cite[Theorem~11.34]{Rotman2009}). We have 
\[H_{1,R}^{\varphi|_{G_1},\alpha_1}\otimes_R R\cong H_{1,R}^{\varphi|_{G_1},\alpha_1}= \Tor_0^R(H_{1,R}^{\varphi|_{G_1},\alpha_1},R) \leqslant H_{1,R}^{\varphi,\alpha} \]
as $R$-modules. We conclude that
\[
\dim_F F \otimes_R H_{1,R}^{\varphi,\alpha} = \infty.
\]
Using flatness again we get 
\[
\dim_F H_{1,F}^{\varphi,\alpha}  = \infty,
\]
and hence the first Alexander polynomials  twisted by $\alpha$ over $F$ and over $R$ vanish by \cref{vanishing}.
\end{proof}

\subsection{Products of limit groups are \texorpdfstring{$\TAP_\infty(\FF)$}{TAP\textinfty(F)}}

We say that $G$ is a \emph{limit group} precisely when it is a finitely generated fully residually free group, that is, for any finite subset $X$ of $G$, there is an epimorphism $f\colon G\to F$ which is injective on $X$ and where $F$ is a free group. 

\begin{thm}\label{limit TAP infinity}
Let $\FF$ be a field and let $G=\prod_{i=1}^n G_i$ be a finite product of limit groups.  Then $G$ is in $\TAP_\infty(\FF)$.
\end{thm}
\begin{proof}
By \cite{Wilton2008} limit groups are LERF, and by \cite{Sela2001I} limit groups are finitely presented, and hence $\mathsf{FP}_2(\Z)$; in fact, by \cite[Exercise 13]{BestvinaFeighn2009}, limit groups are of type $\mathsf{F}$.  It follows that products of limit groups are $\TAP_1(\FF)$ by \Cref{LERF implies TAP} and \Cref{cor.products}.

Let $\varphi\colon G\twoheadrightarrow \ZZ$ be a character which is $\mathsf{FP}_{k-1}(\FF)$-semi-fibred but not $\mathsf{FP}_k(\FF)$-semi-fibred for some $2\leqslant k\leqslant n$. If no such $k$ exists, then we are done by \cref{semifibred so Alex}. The same result tells us that
 the twisted Alexander polynomials of $\varphi$ in dimension at most $k-1$ never vanish. We need to exhibit a vanishing one in dimension $k$. \cref{untwisting} tells us that it is enough to find such a vanishing twisted polynomial for some normal finite-index subgroup of $G$.

 We may assume that if some $G_i$ is abelian then $\varphi|_{G_i}=0$.  Otherwise, $\varphi$ would be $\mathsf{FP}_\infty(\FF)$-semi-fibred by Meinert's inequality. 
After passing to a finite index normal subgroup $H\times K$ with $H=\prod_{i=1}^p H_i, K= \prod_{j=1}^q K_j$, $p+q=n$, $H_i\trianglelefteqslant G_i$, and $K_j= G_{q+j}$, we may assume that $\varphi|_{H_i}$ is surjective and $\varphi|_{K_j}=0$.  Let $\psi$ denote the restriction of $\varphi$ to $H$.  By \cite[Theorem~7.2]{BridsonHowieMillerShort2009} (note that the result is only stated for $\QQ$ but by the paragraph after Theorem~C loc.\ cit.\, it holds for arbitrary fields) we have that $H_p(\ker\psi;\FF)$ has infinite dimension over $\FF$ (here we are using the fact that $\psi$ vanishes on abelian factors).  It follows from \Cref{vanishing} that the twisted Alexander polynomial of $G$ in dimension $p$ associated to $\alpha\colon G\twoheadrightarrow G/(H\times K)$ vanishes.  

We have found a vanishing Alexander polynomial in dimension $p$.  Note that $p\geqslant k$ since $\varphi$ is $\mathsf{FP}_{k-1}(\FF)$-semi-fibred. Meinert's inequality tells us that $\Sigma^{p-1}(G;R)^c$ is the union of joins of the form
\[
\Sigma^{m_1}(G_1;R)^c \ast \dots \ast \Sigma^{m_n}(G_n;R)^c
\]
with $\sum m_i = p-1$. Each such join must therefore have at most $p-1$ factors with $m_i >0$, and hence characters lying in such a join must vanish on all but at most $p-1$ factors $G_i$. But $\varphi$ does not vanish on $p$ factors, and hence $\varphi \in  \Sigma^{p-1}(G;R)$. Hence $p-1 \leqslant k-1$, and therefore $p=k$. We have now shown that the first dimension in which a twisted Alexander polynomial vanishes is equal to the first dimension in which $\varphi$ is not semi-fibred. This proves the claim.  
\end{proof}

\begin{rmk}
    It was pointed out by a referee that \Cref{limit TAP infinity} might hold for the more general class of (products of) limit groups over Droms RAAGs, studied in \cite{Casals_Kazachkov2011,LopezGZ22,Casals_EtAl23,Fruchter23,KochloukovaLopezGZ24} (we refer the reader to the previous citations for the relevant definitions).  The authors are happy to report that it does.

    \begin{thm}
    Let $\FF$ be a field and let $G=\prod_{i=1}^n G_i$ be a product of limit groups over centreless Droms RAAGs.  Then $G$ is in $\TAP_\infty(\FF)$.
    \end{thm}
    
    \begin{proof}[Sketch proof.] By \cite[Corollary 9.5]{Casals_EtAl23}, a limit group over a Droms RAAG is type $\mathsf{F}_\infty$. Since finitely generated subgroups of a limit group are themselves limit gorups, it follows that such subgroups are finitely presented. Hence \cite[Theorem~10.8]{LopezGZ22} implies that a limit group over a Droms RAAG is LERF.  The remainder of the argument follows \Cref{limit TAP infinity} verbatim, except we replace the use of \cite[Theorem~7.2]{BridsonHowieMillerShort2009} with \cite[Theorem~7.3]{LopezGZ22} and note that the latter result holds for arbitrary fields (checking this is laborious, but not hard).\end{proof}
    \end{rmk}

\section{Profinite rigidity of fibring}

\subsection{Cohomological preliminaries}
The goal of this subsection is to establish the relationship between the cohomology of a group and of its profinite completion.

\begin{dfn}
	Let $G$ be a group, $R$ be a ring, and let $\calc$ be a directed system of normal finite-index subgroups of $G$. We set
	\[
	\widehat{G}_\calc = \varprojlim_{U\in\calc} G/U
	\]
	and
		\[
	\cgr R G_\calc = \varprojlim_{U\in\calc} R(G/U).
	\]
	
	When $\calc$ consists of all normal subgroups of finite index, we write $\widehat{G}$ for $\widehat{G}_\calc$, $\cgr R G$ for $\cgr R G_\calc$, and call them respectively the \emph{profinite completion} and the \emph{completed group ring}. 
\end{dfn}

Note that $\widehat \Z$ is a ring with the obvious multiplication.

The groups $\widehat G$ and more generally $\widehat G_\calc$ carry a natural compact topology obtained as the limit of the discrete topology on $G/U$. Whenever we will use this topology, we will state it explicitly, as we do below.

\begin{dfn}
	Let $G$ be a residually finite group. We say that $G$ is \emph{$n$-good} if for all $0\leqslant j\leqslant n$ and all $\ZZ G$-modules $M$ that are finite as sets, the map
	\[H^n_{\mathrm{cont}}(\widehat{G};M)\to H^n(G;M) \]
	induced by the inclusion  $G\to\widehat{G}$ is an isomorphism.  Here, $H^\ast_{\mathrm{cont}}$ denotes \emph{continuous group cohomology} which is defined analogously to ordinary group cohomology except for the following modifications: First, we require $M$ to be a topological $\widehat G$-module, that is, $M$ carries a (possibly discrete) topology and the $\widehat G$-action on $M$ is continuous, and secondly, the cochain groups $C^\bullet_{\mathrm{cont}}(\widehat G;M)$ consist of continuous functions $\widehat{G}^n\to M$.  
	
	A group that is $n$-good for all $n$ is called \emph{cohomologically good}, or \emph{good in the sense of Serre}.
\end{dfn}

\begin{rmk}
	\label{1-good}
	It is very easy to see that every residually finite group is $1$-good.
\end{rmk}

\begin{prop}[{\cite[Lemma 3.2]{Grunewaldetal2008}}]
	Finite-index subgroups of $n$-good groups are themselves $n$-good.
	\label{good is virtual}
\end{prop}

The above proposition is stated in a slightly less general way in the paper of Grunewald--Jaikin-Zapirain--Zalesskii~\cite{Grunewaldetal2008}, but the proof gives precisely what we claim above.

The following result is a slight variation on a theorem of Kochloukova and Zalesskii.  The only difference consists of replacing the assumption of $G$ being type $\mathsf{FP}_\infty$ with the assumption of $G$ being type $\mathsf{FP}_n$.  The proof is very similar but we include it to highlight the differences.

\begin{prop}\emph{\cite[Theorem~2.5]{KochloukovaZalesskii2008}}\label{prop.directedFPn}
Let $G$ be a group of type $\mathsf{FP}_n(\Z)$ and let $\calc$ be a directed system of finite index normal subgroups.  Suppose that for a fixed prime $p$ and for all $1\leqslant i \leqslant n$ we have
\[\varprojlim_{U\in\calc} H_i(U;{\Z/ p\Z})=0. \]
Then, for all $m\geqslant 1$ and $1\leqslant i \leqslant n$ we have
\[ \Tor^{\ZZ G}_i(\ZZ,\cgr{(\Z/ p^m\Z)}{G}_\calc)=0 \textrm{ and } \Tor^{\ZZ G}_i(\ZZ, \cgr{\Z_p} G_\calc)=0\]
where $\Z_p$ denotes the $p$-adic integers.
\end{prop}
In both the statement above and the proof below, we stay in the abstract category, that is we do not require any continuity, and homology is taken without closing images.
\begin{proof}	
Let $P_\bullet$ be a projective resolution of $\ZZ$ over $\ZZ G$ such that $P_i$ is finitely generated for $i\leqslant n$.  Let 
${P}^{(m)}_\bullet = {\cgr{(\Z/ p^m\Z)}{G}_\calc}\otimes_{\Z G} P_\bullet$.  By \cite[Lemma~2.1]{KochloukovaZalesskii2008} we have \[H_i({P}^{(1)}_\bullet)\cong\Tor^{\ZZ G}_i(\ZZ,{\cgr{(\Z/ p \Z)}{G}_\calc})=0 \text{ for } 1\leqslant i\leqslant n.\]

The short exact sequence of right $\ZZ G$-modules
\[\begin{tikzcd}[column sep=0.8cm]
0 \arrow[r] & \cgr{(\Z/ p\Z)}{G}_\calc \arrow[r] & \cgr{(\Z/ p^m\Z)}{G}_\calc \arrow[r] & \cgr{(\Z/ p^{m-1}\Z)}{G}_\calc \arrow[r] & 0
\end{tikzcd} \]
induces a long exact sequence in homology containing sequences 
\[\begin{tikzcd}[column sep=0.3cm]
 H_i(G;\cgr{(\Z/ p\Z)}{G}_\calc) \arrow[r] & H_i(G;\cgr{(\Z/ p^m\Z)}{G}_\calc) \arrow[r] &  H_i(G;\cgr{(\Z/ p^{m-1}\Z)}{G}_\calc) 
\end{tikzcd} \]
exact in the middle term.
This latter sequence implies via an easy induction that 
\[
\Tor_i^{\ZZ G}(\ZZ,\cgr{(\Z/ p^m\Z)}{G}_\calc)=0 \text{ for } 1\leqslant i\leqslant n,
\] 
and so $P^{(m)}_\bullet$ is exact up to dimension $n$. It also shows that 
\[
H_{n+1}(G;\cgr{(\Z/ p^m\Z)}{G}_\calc) \to  H_{n+1}(G;\cgr{(\Z/ p^{m-1}\Z)}{G}_\calc)
\] is a surjection.

For every $m$ we have an obvious chain map $P^{(m+1)}_\bullet \to P^{(m)}_\bullet$. 
Let $Q_\bullet=\varprojlim_{m} P^{(m)}_\bullet$ where the limit is taken along these maps.  By \cite[Proposition 3.5.7 and Theorem~3.5.8]{Weibel1994}, the complex $Q_\bullet$ is exact up to dimension $n$ and by construction $Q_\bullet\cong \cgr{\Z_p}{G}_\calc \otimes_{\Z G} P_\bullet$. Therefore, 
\[
H_i(Q_\bullet)\cong\Tor_i^{\ZZ G}(\ZZ;\cgr{\Z_p}{G}_\calc)=0 \text{ for }1\leqslant i \leqslant n
\]
as required.
\end{proof}

The next result is due to Jaikin-Zapirain; we have weakened the original assumption of type $\mathsf{FP}_\infty$ to $\mathsf{FP}_n$.  The proof goes through verbatim after substituting \Cref{prop.directedFPn} for Jaikin-Zapirain's use of \cite[Theorem~2.5]{KochloukovaZalesskii2008}.

\begin{prop}\label{prop.ngood}\emph{\cite[Proposition~3.1]{JaikinZapirain2020}}
Let $G$ be a group of type $\mathsf{FP}_n(\Z)$ and let $(F_\bullet,\partial_\bullet)$ be a free resolution of the trivial $\ZZ G$-module $\ZZ$ which is finitely generated up to dimension $n$, and in which $F_0 = \Z G$.  Then $G$ is $n$-good if and only if the induced sequence 
\[ \begin{tikzcd}
\cdots \arrow[r,"\widehat{\partial_{n+1}}"] & \widehat{F_n} \arrow[r, "\widehat{\partial_n}"] & \cdots \arrow[r, "\widehat{\partial_2}"] & \widehat{F_1} \arrow[r, "\widehat{\partial_1}"] & \widehat{F_0} \arrow[r, "\widehat{\partial_0}"] & \widehat{\ZZ}
\end{tikzcd}\]
is exact up to dimension $n$, where $(\widehat{F}_\bullet,\widehat{\partial}_\bullet)$ is obtained from $(F_\bullet,\partial_\bullet)$ by tensoring with $\cgr{\widehat{\Z}} G$ over $\Z G$. 
\end{prop}

\subsection{Towards profinite fibring}
In this section our goal is to set up a correspondence between the characters of two profinitely isomorphic groups.  The key tool will be the $\epsilon$-pullbacks defined below which set up this `correspondence'.  We also recall a technical result of Liu which we will use later. 

First we need to introduce some notation. 
Recall that
\[
H_{n,R}^{\varphi,\alpha} = H_n\big( G; R(Q \times Z)\big)
\]
 where $\alpha\colon G\onto Q$ and $\varphi \colon G \to Z$ are homomorphisms, and $R(Q \times Z)$ is a right $RG$-module via $(q,z).g = (q\alpha(g), z\varphi(g))$ with $(g,q,z) \in G \times Q\times Z$. We also treat $R(Q \times Z)$ as an $RZ$ module via the inclusion $Z \to Q \times Z$.  

Now suppose that $Z \in \{ \Z, \widehat \Z\}$, so that $\widehat Z = \widehat \Z$. Let $\widehat G$ be the profinite completion of $G$, and let $\widehat \alpha \colon \widehat G \to Q$ and $\widehat \varphi \colon \widehat G \to \widehat \Z$ be the completions of the morphisms from before. Note that $Q = \widehat Q$ since $Q$ is finite. Let $R = \FF$ be a finite field. We let
\[
\widehat{H}_{n,\FF}^{\widehat \varphi,\widehat \alpha} = H_n^{\mathrm{prof}}(\widehat G; \cgr{\FF}{Q \times \widehat\Z})
\]
where $H_\ast^{\mathrm{prof}}$ denotes \emph{profinite homology}, as defined in \cite[\S6.3]{RibesZalesskii2010}. Observe that $\cgr{\FF}{Q \times \widehat\Z} = \cgr{\FF Q}{\widehat\Z}$ has a structure of an  $\cgr{\FF }{\widehat\Z}$ module, and hence so does $\widehat{H}_{n,\FF}^{\widehat \varphi,\widehat \alpha}$.

We now recall the technical result of Liu we need.

\begin{prop}\emph{\cite[Proposition~4.6]{Liu2020}}
\label{prop.ann.Liu}
	Let $G$ be a group which is 
	 $n$-good and of type $\mathsf{FP}_n(\Z)$.  Let $\FF$ be a finite field.  Let $\alpha\colon G\twoheadrightarrow Q$ be a finite quotient of $G$.  Denote by $\widehat{\alpha}\colon \widehat{G}\twoheadrightarrow Q$ the completion of $\alpha$.
	\begin{enumerate}
	\item \label{prop.ann.Liu1}
	Let $\varphi\colon  G\to \Z$ be a group homomorphism, and let $\widehat \varphi \colon \widehat G \to \widehat \Z$ denote its completion.
	If the	annihilator of ${H}_{n,\FF}^{ \varphi, \alpha}$ in $\FF \widehat \Z$ is non-zero, then the annihilator of  $\widehat{H}_{n,\FF}^{\widehat \varphi,\widehat \alpha}$
	is non-zero in $\cgr{\mathbb{F}}{\widehat\ZZ}$.
	
	\item \label{prop.ann.Liu2}
	Let $\varphi,\psi\colon G\to\widehat{\ZZ}$ be group homomorphisms and suppose that $\ker(\psi)$ contains $\ker(\varphi)$. If $H_{n, \FF}^{\psi,\alpha} $ has a non-zero annihilator in $\FF \widehat \Z$, then $H_{n,\FF}^{\varphi,\alpha}$ has a non-zero annihilator in $\FF \widehat \Z$.
	
	\item \label{prop.ann.Liu3}
	Let $\Gamma$ be a profinite group, let $\Psi\colon\Gamma\to\widehat{G}$ be a continuous epimorphism and let $\psi\colon  G\to \widehat{\ZZ}$ be a group homomorphism.  Let $\widehat{\alpha}'$ and $\widehat{\psi}'$ denote the pullbacks $\widehat\alpha\circ\Psi$ and $\widehat{\psi}\circ\Psi$. If $\widehat{H}_{n,\FF}^{\widehat \varphi',\widehat \alpha'}$ has a non-zero annihilator in  $\cgr \FF {\widehat\Z}$, then $\widehat{H}_{n,\FF}^{\widehat \varphi,\widehat \alpha}$ has a non-zero annihilator in $\cgr \FF {\widehat\Z}$.
	
	\item \label{prop.ann.Liu4}
	Let $\varphi\colon G\to\ZZ$ be a group homomorphism. The module $\widehat{H}_{n,\FF}^{\widehat \varphi,\widehat \alpha}$ has a non-zero annihilator in $\cgr \FF {\widehat\Z}$ if and only if ${H}_{n,\FF}^{\varphi,\alpha}$ has finite dimension over $\mathbb{F}$.
	\end{enumerate}
\end{prop}
Note that we have weakened the hypotheses `cohomologically good and type $\mathsf{FP}_\infty$' in \cite[Proposition~4.6]{Liu2020} to `$n$-good and type $\mathsf{FP}_n$'.  To make the adjustment we simply substitute the use of \cite[Proposition~3.1]{JaikinZapirain2020} in the proof of \cite[Proposition~4.6]{Liu2020} with our \Cref{prop.ngood}. 

\begin{defn}
Let $H_A$ and $H_B$ be a pair of finitely generated $\ZZ$-modules.  Let $\Phi \colon \widehat{H_A}\to\widehat{H_B}$ be a continuous homomorphism of the profinite completions.  We define the \emph{matrix coefficient module} 
\[\MC(\Phi;H_A,H_B)\]
 (or simply 
$\MC(\Phi)$ if there is no chance of confusion) for $\Phi$ with respect to $H_A$ and $H_B$ to be the smallest $\ZZ$-submodule $L$ of $\ZZhat$ such that $\Phi(H_A)$ lies in the submodule $H_B\otimes_\ZZ L$ of $\widehat{H_B}$.  We denote by 
\[\Phi^\MC\colon H_A\to H_B\otimes_\ZZ\MC(\Phi)\]
 the homomorphism uniquely determined by the restriction of $\Phi$ to $H_A$.
\end{defn}

By \cite[Proposition 3.2(1)]{Liu2020}, the $\Z$-module $\MC(\Phi;H_A,H_B)$ is a non-zero finitely generated free $\Z$-module.

Given two profinitely isomorphic groups, the next definition will give us a way to pullback homomorphisms to $\ZZ$ from one group to the other through their (shared) profinite completion.  When the $\TAP_n(\FF)$ property holds we will be able to verify whether the characters are fibred.  The purpose of the $\epsilon$ we define is to construct this pullback.

\begin{defn}
	\label{epsilon}
We define 	$\epsilon\in\Hom_\ZZ(\MC(\Phi),\Z)$ by picking a free basis for $\MC(\Phi)$ and sending every generator to either $0$ or $1$ in such a way that following $\epsilon$ with the natural projection $\Z \to \Z/2\Z$ coincides with the restriction of the natural projection $\widehat \Z \to \Z/2 \Z$ to $\MC(\Phi)$. The definition of $\epsilon$ depends on the choice of a basis for $\MC(\Phi)$.
	
The \emph{$\epsilon$-specialisation} of $\Phi$ refers to the composite homomorphism
\[\begin{tikzcd}
H_A \arrow[r, "\Phi^\MC"] & H_B\otimes_\ZZ\MC(\Phi) \arrow[r, "1\otimes\epsilon"] & H_B\otimes_\ZZ \Z = H_B,
\end{tikzcd}\]
denoted by $\Phi_\epsilon\colon H_A \to H_B$.  The \emph{dual $\epsilon$-specialisation} of $\Phi$ refers to the homomorphism $\Phi^\epsilon \colon \Hom_\Z(H_B, \Z) \to \Hom_\Z(H_A, \Z)$ precomposing with $\Phi_\epsilon$. 
\end{defn}

\begin{lem}
	\label{epsilon pullbacks}
If $\Phi$ is an isomorphism, then the images of $\Phi_\epsilon$ and $\Phi^\epsilon$ are of finite index in their respective codomains.
\end{lem}
\begin{proof}
Let $b$ denote the rank of $H_B$. We have a natural epimorphism $\rho \colon H_B \to (\Z / 2\Z)^b$ that extends to $\widehat \rho \colon H_B \otimes_\Z \widehat \Z \to (\Z/2\Z)^b$. By construction, $\rho \circ \Phi_\epsilon =  \widehat \rho \circ \Phi$. Let us assume that $\Phi$ is an isomorphism.
 Since $\widehat \rho$ is clearly surjective, we conclude that   $\rho \circ \Phi_\epsilon$ is surjective. Pick a basis of $(\Z/2\Z)^b$, and lift it via $\rho$ to a set $v_1, \dots, v_b \in \im \Phi_\epsilon$. Suppose that the elements $v_1, \dots, v_b$ are $\Z$-linearly dependent. By removing the common factors of $2$ from the coefficients, we may assume that we have
 \[
 \sum_{i=1}^b \lambda_i v_i =0
 \]
 with $\lambda_i \in \Z$ and with at least one $\lambda_i$ odd. Applying $\rho$ to this formula contradicts the fact that we started with a basis for $(\Z/2\Z)^b$. Hence $v_1, \dots, v_b$ are $\Z$-linearly independent, and hence by tensoring with $\QQ$ we see that $\im \Phi_\epsilon$ is of finite index in $H_B$.
 
 The result for $\Phi^\epsilon$ follows immediately, since we have just shown that $\Phi_\epsilon \otimes_\Z \id_\QQ$ is surjective, and hence an isomorphism, since $H_A$ and $H_B$ have the same rank.
\end{proof}

\begin{defn}
Let $G_A$ and $G_B$ be finitely generated groups and let $\Psi\colon \widehat{G}_A\to \widehat{G}_B$ be an isomorphism of profinite completions. Let $H_A$ and $H_B$ be the maximal torsion-free quotients of the abelianisations of, respectively,  $G_A$ and  $G_B$; let $\mathrm{ab}$ denote both of the free abelianisation maps. Note that $\Psi$ induces $\Psi_1  \colon \widehat{H_A} \to \widehat{H_B}$.
Pick  $\epsilon\in\Hom_\ZZ(\MC(\Psi_1),\ZZ)$ as in \cref{epsilon}.  Given $\varphi \in H^1(G_B;\Z)$ we define 
\[\psi= \Psi_1^\epsilon(\varphi\circ \mathrm{ab}^{-1}) \circ \mathrm{ab}\in H^1(G_A;\ZZ)\]
 to be the \emph{$\epsilon$-pullback} of $\varphi$.
\end{defn}

\subsection{The result}

\begin{thm}\label{prop.profinite.fibre}
	Let $n$ be a positive integer.
	Let $G_A$ and $G_B$ be $n$-good groups of type $\mathsf{FP}_n(\Z)$, and suppose that $G_B$ is in $\TAP_n(\FF)$, where $\FF$ is a finite  field.  Let $\Psi\colon\widehat{G}_A\to \widehat{G}_B$ be an isomorphism of profinite completions and let $\varphi\in H^1(G_B;\ZZ)$.  If for every $i \leqslant n$ an $\epsilon$-pullback $\psi\in H^1(G_A;\ZZ)$ of $\varphi$ has non-vanishing $i$th twisted  Alexander polynomials over $\FF$, then $\varphi$ is  $\mathsf{FP}_n(\FF)$-semi-fibred.
\end{thm}
\begin{proof}
	Note that $\Psi$ is continuous by the work of Nikolov--Segal \cite[Theorem~1.1]{NikolovSegal2007I} (see also \cite{NikolovSegal2007II} for the remainder of the proof). Let $\widehat\rho\colon G_A\to\widehat\ZZ$ denote the composite \[G_A\rightarrowtail \widehat{G}_A\xrightarrow{\Psi}\widehat{G}_B\xrightarrow{\widehat\varphi}\widehat\ZZ,\]
	 where $\widehat \varphi$ is the completion of $\varphi$.  Observe that $\Ker(\psi)$ contains $\Ker(\hat\rho)$. Indeed, $\hat\rho$ factorises as the top composite and $\psi$ as the bottom composite
	\[
	\begin{tikzcd}[column sep=0.4cm]
	& & & & & \ZZhat \\
	G_A \arrow[r] & H_A \arrow[r, "{\Psi_1}^\MC"] & H_B\otimes_\ZZ\MC(\Psi_1) \arrow[r, "\varphi\otimes 1"] & \ZZ\otimes_\ZZ\MC(\Psi_1) \arrow[r, "="] & \MC(\Psi_1) \arrow[ru, tail] \arrow[rd, "\epsilon"] &        \\
	& & & & & \ZZ,
	\end{tikzcd}
	\]
	so clearly $\psi$ vanishes on everything $\widehat\rho$ vanishes on.
	
	Let $\beta\colon G_B\twoheadrightarrow Q$  be a finite quotient with completion $\widehat \beta$, and let $\alpha\colon G_A\twoheadrightarrow Q$ denote the composite $G_A\rightarrowtail \widehat{G}_A\xrightarrow{\Psi}\widehat{G}_B\xrightarrow{\widehat\beta}Q$. Let $i\leqslant n$. By assumption, the homology group $H_{i,\FF}^{\psi,\alpha}$ is $\FF \Z$-torsion, and hence
	\[
	0 = \Frac( \FF \Z) \otimes_{\FF \Z} H_{i,\FF}^{\psi,\alpha}  = H_i(G_A; \Frac( \FF \Z)Q)
	\]
	for $i \leqslant n$, where the second equality comes from the fact that localisations are flat, and that $\Frac( \FF \Z)Q$ is the localisation of $\FF (\Z \times Q)$ at $\FF (\Z \times \{1\}) \smallsetminus \{0\}$.
	
	Since $G_A$ is of type $\mathsf{FP}_n(\Z)$, we  find a free resolution $C_\bullet$ of $\Z$ with each $C_i$, for $i\leqslant n$, a finitely generated $\Z G_A$-module; let $\partial_i \colon C_i \to C_{i-1}$ denote the differentials of $C_\bullet$.  The fact that $H_i(G_A; \Frac( \FF \Z)Q) = 0$ for all $i \leqslant n$ allows us to construct chain contractions, that is, $\Frac( \FF \Z)Q$-module maps 
	\[d_i \colon \Frac( \FF \Z)Q \otimes_{\Z G_A} C_i \to \Frac( \FF \Z)Q \otimes_{\Z G_A} C_{i+1}\]
	 for $i\leqslant n$ with
	\[
	d_{i-1} \circ \partial_i + \partial_{i+1} \circ d_i = \id,
	\]
	where we now view $\partial_i$ as $\id_{\Frac( \FF \Z)Q} \otimes \partial_i$ -- for details on how to build the chain contractions, see \cite[Section I.7]{Brown1982} or \cite[Section 2.2]{Turaev2001}. Since the modules \[\Frac( \FF \Z)Q \otimes_{\Z G_A} C_i\] are finitely generated, by multiplying the maps $d_i$ by the common denominator of all the entries of the matrices representing the maps $d_i$, we arrive at the existence of $\FF (\Z \times Q)$-module maps 	
	\[d'_i \colon \FF (\Z \times Q) \otimes_{\Z G_A} C_i \to \FF (\Z \times Q) \otimes_{\Z G_A} C_{i+1}\]
	with
	 	\[
	 d'_{i-1} \circ \partial_i + \partial_{i+1} \circ d'_i
	 \]
	 being equal to the right-multiplication by some \[z \in \FF (\Z \times \{1\}) \smallsetminus \{0\}.\]
	 Again, we have to interpret the differentials $\partial_i$ in a suitable way. Crucially, since $\FF \Z$ is central in $\FF (\Z \times Q)$, right-multiplication by $z$ coincides with left-multiplication by $z$.
	
	Let $ \psi' \colon G_A \to \widehat \Z$ denote $\psi$ followed by the natural embedding $\Z \to \widehat \Z$. The maps $d_i'$ can be easily extended to maps 
	\[\FF ( \widehat \Z \times Q) \otimes_{\Z G_A} C_i \to \FF (\widehat \Z \times Q) \otimes_{\Z G_A} C_{i+1}\]
	immediately yielding that 
 $H_{i,\FF}^{\psi',\alpha}$ is $\FF  \Z$-torsion, and hence $\FF \widehat \Z$-torsion. Still, $\ker (\widehat \rho) \leqslant \ker (\psi')$. Applying \Cref{prop.ann.Liu}\eqref{prop.ann.Liu2}, \eqref{prop.ann.Liu1}, \eqref{prop.ann.Liu3}, and \eqref{prop.ann.Liu4} in the given order, we see that $H_{i,\FF}^{\varphi,\beta}$ is a finite dimensional $\FF$-module, and hence a torsion $\FF \Z$-module.   Since $\beta$ was arbitrary and $G_B\in\TAP_n(\FF)$, it follows that $\varphi$ is $\mathsf{FP}_n(\FF)$-semi-fibred.
\end{proof}

\begin{cor}\label{Prop.pro.higherdim}
	Let $n$ be a positive integer.
	Let $G_A$ and $G_B$ be $n$-good groups of type $\mathsf{FP}_n(\Z)$  with isomorphic profinite completions. Suppose that $G_A$ lies in $\TAP_n(\FF)$, where $\FF$ is a finite  field. The group $G_A$ is $\mathsf{FP}_n(\FF)$-semi-fibred if $G_B$ is.
\end{cor}
\begin{proof}
	Let $\psi \colon G_B \to \Z$ be a non-trivial $\mathsf{FP}_n(\FF)$-semi-fibred character; observe that this statement remains unchanged if we replace $\psi$ by a positive scalar multiple. By \cref{semifibred so Alex}, the twisted Alexander polynomials of $\psi$ over $\FF$ never vanish. \cref{epsilon pullbacks} gives us a bijection between positive scalar multiples of characters in $H^1(G_A;\ZZ)$ and $H^1(G_B;\ZZ)$, and hence, in particular, we find a non-trivial character $\varphi \colon G_A \to \Z$ such that $\psi$ is its $\epsilon$-pullback (up to multiplication by a positive scalar). \cref{prop.profinite.fibre} shows that $\varphi$ is $\mathsf{FP}_n(\FF)$-semi-fibred.
\end{proof}

We may summarise the above by saying that being $\mathsf{FP}_n(\FF)$-semi-fibred is a profinite property among $n$-good groups of type $\mathsf{FP}_n(\Z)$ in $\TAP_n(\FF)$.

Using \cref{1-good} we obtain the following crisper formulation for $n=1$.

\begin{cor}
	\label{dim 1}
Let $G_A$ and $G_B$ be finitely generated groups  with isomorphic profinite completions. Suppose that $G_B$ lies in $\TAP_1(\FF)$, where $\FF$ is a finite  field. If $G_A$ is algebraically semi-fibred, then so is $G_B$.
\end{cor}

\section{Applications}

\subsection{Products of LERF groups}

\begin{thm}\label{thm.profinite.LERF}
Let $G_A$ and $G_B$ be groups such that all of the following hold:
\begin{itemize}
	\item $G_A$ is finitely generated;
	\item $G_B$ is a product of LERF groups and is of type $\mathsf{FP}_2(R)$ for some commutative ring $R$;
	\item there is an isomorphism $\widehat {G_B} \to \widehat {G_A}$.
\end{itemize}
	If $G_A$ is algebraically semi-fibred, then $G_B$ is algebraically fibred.
\end{thm}
\begin{proof}
The group $G_B$ is in $\TAP_1(\FF)$ for every finite  field $\FF$ by \cref{LERF implies TAP} and \cref{cor.products} -- we are also using the fact that each of the factors of $G_B$ is itself of type $\mathsf{FP}_2(R)$, which is easy to see. Now we use \cref{dim 1} and see that $G_B$ is algebraically semi-fibred.  But the first BNS invariant of LERF groups is symmetric by \cref{LERF symmetric BNS}.  It follows from Meinert's inequality that products of LERF groups also have symmetric first BNS invariant, and hence that $G_B$ is algebraically fibred.
\end{proof}

The following is restating \cref{prod limit profinite intro} from the introduction.

\begin{thm}
	\label{prod limit profinite}
Let $\FF$ be a finite  field.  Let $G_A$ and $G_B$ be profinitely isomorphic finite products of limit groups.  The group $G_A$ is $\mathsf{FP}_n(\FF)$-semi-fibred if and only if $G_B$ is.
\end{thm}
\begin{proof}
By \Cref{limit TAP infinity}, finite products of limit groups are $\TAP_\infty(\FF)$; they are also of type $\mathsf F$, as mentioned before. The result now follows from  \Cref{Prop.pro.higherdim}. Indeed, limit groups are cohomologically good by \cite[Theorem~1.3]{Grunewaldetal2008} and so products of them are cohomologically good by \cite[Theorem~2.5]{Lorensen2008}.
\end{proof}

\subsection{Poincar\'e duality groups}

We now turn our attention to $\mathsf{PD}_{3}$-groups, that is, Poincar\'e duality groups in dimension $3$. For an introduction to this topic, see \cite{Hillman2020}.

\begin{thm}\label{thm.PD3-3manifold.profinite}
Let $G_A$ be a $\mathsf{PD}_3$-group in $\TAP_1(\FF)$ for some finite  field $\FF$. Let $G_B$ be a finitely generated algebraically fibred group.
If $\widehat{G_A}\cong\widehat{G_B}$, then $G_A$ is the   
fundamental group of a closed connected $3$-manifold $M$.  Moreover, $M$ is a mapping torus of a compact surface.  
\end{thm}

\begin{proof}
By \cite[Theorem~5]{Hillman2020b} and \cref{HNN criterion} we have that $\Sigma^1(G_A)=-\Sigma^1(G_A)$.  Indeed, Hillman's result tells us that any ascending HNN extension splitting of $G$ with finitely generated base group $N$ must be a semidirect product $G \cong N\rtimes\Z$.   By \cref{dim 1}, $G_A$ is algebraically fibred. Hence, 
\[
G_A = K \rtimes \Z 
\]
for some finitely generated subgroup $K$.
It follows from a result of Strebel \cite{Strebel1977} (see \cite[Theorem~1.19]{Hillman2002} for an explanation), that $K$ has cohomological dimension at most $2$ and hence is a $\mathsf{PD}_2$-group.   In particular, by \cite{EckmannMueller1980} (see also \cite{KielakKropholler2021}) the group $K$ is isomorphic to the fundamental group of a closed surface.  Since every outer automorphism of $K$ is realised by a mapping class of the underlying surface by the Dehn--Nielsen--Baer theorem, we conclude that $G_A$ is the  fundamental group of a closed connected $3$-manifold $M$, namely the mapping torus of a compact surface with fundamental group $K$.
\end{proof}

The following is restating \cref{pd3 intro} from the introduction.

\begin{corollary}
	Let $G_A$ be a   LERF  $\mathsf{PD}_3$-group. Let $G_B$ be the fundamental group of a closed  connected hyperbolic $3$-manifold.
If $\widehat{G_A}\cong\widehat{G_B}$, then $G_A$ is the   
fundamental group of a closed connected hyperbolic $3$-manifold.
\end{corollary}
\begin{proof}
As in the proof of the previous theorem, by \cite[Theorem~5]{Hillman2020b} and \cref{HNN criterion}, for every finite index subgroup $G_A'\leqslant G_A$ we have that $\Sigma^1(G_A')=-\Sigma^1(G_A')$.  Let $H_B$ be a finite index subgroup of $G_B$ that is algebraically fibred -- the existence of such a subgroup is guaranteed by Agol's theorem \cite{Agol2013}. Let $H_A$ be the corresponding finite index subgroup of $G_A$; we still have $\widehat{H_A} \cong \widehat{H_B}$. The group $H_A$ is still a $\mathsf{PD}_3$-group by \cite[Theorem 2]{JohnsonWall1972}. It is immediate that $H_A$ is LERF. Since all  $\mathsf{PD}_3$-groups are of type $\mathsf{FP}(\Z)$, we conclude, using \cref{LERF implies TAP}, that $H_A$ is $\TAP_1(\FF)$ for every finite  field.  \cref{thm.PD3-3manifold.profinite} now shows that $H_A$ is the fundamental group of a connected compact $3$-manifold. By  \cite[Lemma~8.2]{Hillman2020}, the group $G_A$ is also a fundamental group of a connected compact $3$-manifold $M$.
The manifold $M$ is hyperbolic by \cite{WiltonZalesskii2017}.
\end{proof}

\subsection{RFRS groups and agrarian Betti numbers}

The following definition is due to Agol \cite{Agol2008} and played a crucial role in solving the Virtual Fibring Conjecture for hyperbolic $3$-manifolds.

\begin{defn}
Let $G$ be a group.  We say that $G$ is \emph{residually finite rationally solvable (RFRS)} if there is a chain of finite index normal subgroups \[G=G_0\geqslant G_1\geqslant G_2 \geqslant \cdots\]
     of $G$ such that
\begin{enumerate}
    \item $\bigcap_\NN G_i=\{1\}$;
    \item $\ker\big( G_i\to H_1(G_i; \QQ) \big) \leqslant G_{i+1}$ for $i\geqslant 0$.
\end{enumerate}
\end{defn}

\begin{defn}
A group $G$ is \emph{indicable} if $G$ is trivial or admits an epimorphism to $\ZZ$.  We say that $G$ is \emph{locally indicable} if every finitely generated subgroup of $G$ is indicable.  
\end{defn}

Note that a subgroup of a RFRS group is RFRS and that RFRS groups are indicable.  Hence, RFRS groups are locally indicable.

\begin{defn}
Let $R$ and $\cald$ be skew-fields, let $G$ be a locally indicable group, and let $\psi\colon RG\to \cald$ be a ring homomorphism.  The pair $(\cald,\psi)$ is \emph{Hughes-free} if
\begin{enumerate}
    \item $\cald$ is generated by $\psi(RG)$ as a skew-field, that is, $\langle\psi(RG)\rangle=\cald$;
    \item for every finitely generated subgroup $H\leqslant G$, every normal subgroup $N\triangleleft H$ with $H/N\cong \ZZ$, and every set of elements $h_1,\dots,h_n\in H$ lying in distinct cosets of $N$, the sum
    \[\langle \psi(RN)\rangle\cdot\psi(h_1)+\dots+\langle \psi(RN)\rangle\cdot\psi(h_n) \]
    is direct.
\end{enumerate}
By Ian Hughes \cite{Hughes1970}, for fixed $R$ and $G$, if such a pair $(\cald,\psi)$ exists, then $\cald$ is unique up to $RG$-algebra isomorphism.  In this case we denote $\cald$ by $\cald_{RG}$.
\end{defn}
(Like the property, the Hughes mentioned here and the first author are free of any of relation.)

The following result is due to Jaikin-Zapirain.

\begin{prop}\emph{\cite[Corollary~1.3]{Jaikin2021}}
If $G$ is a RFRS group and  $R$ is a skew-field,  then $\cald_{RG}$ exists and it is the universal division ring of fractions of $RG$.
\end{prop}

\begin{defn}
A group $G$ is \emph{agrarian} over a ring $R$ if there exists a skew-field $\cald$ and a monomorphism $\psi\colon RG \rightarrowtail \cald$ of rings.  If $G$ is agrarian over $R$, then we define the \emph{agrarian $\cald$-homology} to be
\[H_j^\cald(G) =  \Tor^{RG}_j(R,\cald)\]
where $R$ is the trivial $RG$-module and $\cald$ is viewed as an $\cald$-$RG$-bimodule via the embedding $RG\rightarrowtail \cald$.  Since modules over a skew-field have a canonical dimension function taking values in $\NN\cup\{\infty\}$ we may define the \emph{agrarian $\cald$-Betti number} by
\[b_j^{\cald}(G) =  \dim_\cald H_j^\cald(G). \]
When $G$ is RFRS, by the previous proposition, we have (up to $RG$-isomorphism) a canonical choice $\cald_{RG}$ of $\cald$ for each skew-field $R$.
\end{defn}

\begin{thm}
	\label{thm.Fisher}
Let $R$ be a skew-field and let $n \in \NN$. Let $G$ be a virtually RFRS group of type $\mathsf{FP}_n(R)$.  The following are equivalent:
\begin{enumerate}
    \item $b_j^{\cald_{RG}}(G)=0$ for all $j\leqslant n$; \label{fisher 1}
    \item $G$ is virtually $\mathsf{FP}_n(R)$-fibred; \label{fisher 2}
    \item $G$ is virtually $\mathsf{FP}_n(R)$-semi-fibred. \label{fisher 3}
\end{enumerate}
\end{thm}
\begin{proof}
	The equivalence of the first two items is \cite[Theorem~6.6]{Fisher2024}. The implication \eqref{fisher 2} $\Rightarrow$ \eqref{fisher 3} is clear, so let us prove \eqref{fisher 3} $\Rightarrow$ \eqref{fisher 1}.
	
	By \cite[Lemma 6.3]{Fisher2024}, the numbers $b_j^{\cald_{RG}}(G)$ scale with the index when passing to finite-index subgroups.  Thus, we may assume without loss of generality that $G$ itself is $\mathsf{FP}_n(R)$-semi-fibred. In particular, let $\varphi \in \Sigma^n(G;R)$ witness this semi-fibration. By \cite[Lemma 5.3]{Fisher2024}, we have
	\[
	\mathrm{Tor}_j^{RG}(R,\nov R G \varphi) =0
	\]
	for all $0 \leqslant j \leqslant n$. 
	
	Let $\mathbb K$ be the skew-field of twisted Laurent series with variable $t$ and coefficients in the skew-field $\cald_{R(\ker \varphi)}$; the variable $t$ is an element of $G$ with $\varphi(t) =1$, a generator of $\Z$, and the twisting extends the conjugation action of $t$ on $\ker \varphi$ to  $\cald_{R(\ker \varphi)}$ -- such an extension is possible since $\cald_{R(\ker \varphi)}$ is Hughes free, see \cite[p.8]{Jaikin2021} for an explanation of this fact. The skew-field $\mathbb K$ contains $\nov R G \varphi$, since the latter can also be viewed as a ring of twisted Laurent series in $t$ with coefficients in $R (\ker \varphi)$, with the twisting described above. Hence, using chain contractions, we see that
	\[
	\mathrm{Tor}_j^{RG}(R,\mathbb K) =0
	\]
		for all $0 \leqslant j \leqslant n$. 
		
Now, Hughes-freeness of  $\cald_{RG}$ tells us that $\cald_{RG}$ is isomorphic as an $R G$-module to the division closure in $\mathbb K$ of the twisted Laurent polynomial ring $R(\ker \varphi)[t^{\pm 1}]$,  where we identify  the rings $R(\ker \varphi)[t^{\pm 1}]$ and $RG$ using the group isomorphism $(\ker \varphi) \rtimes \Z = G$.  This endows $R(\ker \varphi)[t^{\pm 1}]$ with an $R G$-bimodule structure. Hence, we may view $\cald_{RG}$ as a subring of $\mathbb K$, and view $\mathbb K$ as a $\cald_{RG}$-module. Since both $\cald_{RG}$ and $\mathbb K$ are skew-fields, $\mathbb K$ is a flat as a $\cald_{RG}$-module.
We conclude that 
	\[
\mathrm{Tor}_j^{RG}(R,\cald_{RG}) =0
\]
for all $0 \leqslant j \leqslant n$, as claimed.
\end{proof}

\begin{thm}
	\label{detecting betti}
	Let $n\in\NN \cup \{\infty\}$, and let $\FF$ be a  finite  field.  Let $G_A$ and $G_B$ be $n$-good virtually RFRS groups of type $\mathsf{FP}_{n}(\FF)$ and suppose that $\widehat{G}_A\cong\widehat{G}_B$.  Suppose that every finite-index subgroup of $G_A$ and $G_B$ is in $\TAP_{n}(\FF)$.  We have
	\[
	\min \{ j \leqslant n \mid b^{\cald_{\FF G_A}}_{j}(G_A) \neq 0 \} = \min \{ j \leqslant n \mid b^{\cald_{\FF G_B}}_{j}(G_B) \neq 0 \}
	\]
	where we take the minimum of the empty set to be $\infty$.
\end{thm}
\begin{proof}
	We first assume that $n \in \NN$.
	Since we are concerned with virtual properties we may assume without loss of generality that $G_A$ and $G_B$ are RFRS, $n$-good, of type $\mathsf{FP}_{n}(\Z)$, and all finite-index subgroups of $G_A$ and $G_B$ are in $\TAP_{n}(\FF)$; we have used \cref{good is virtual} here.  
	
	Suppose that $b^{\cald_{\FF G_A}}_{j}(G_A)=0$ for $j\leqslant m$ for some $m \leqslant n$. The group $G_A$ is virtually $\mathsf{FP}_m(\FF)$-fibred by \cref{thm.Fisher}. We may pass to further finite index subgroups of $G_A$ and $G_B$ and assume that $G_A$ is $\mathsf{FP}_m(\FF)$-fibred.
	By \cref{Prop.pro.higherdim}, the group $G_B$ is $\mathsf{FP}_m(\FF)$-semi-fibred, and hence 
	\[b^{\cald_{\FF G_B}}_{j}(G_B)=0\]
	for $j\leqslant m$ by \cref{thm.Fisher}. This shows an inequality between the minima in the statement. The argument is symmetric in $G_A$ and $G_B$, and hence we also obtain the converse inequality.
	
	\smallskip
	Now suppose that $n = \infty$. If both of the minima in the statement are $\infty$, then we are done. Without loss of generality let us suppose that the left-hand side one is equal to $m < \infty$.
We observe that $G_A$ and $G_B$ satisfy the hypothesis of our theorem for $n = m$, and hence the right-hand side minimum is also equal to $m$. 
\end{proof}

Observe that the above result applies in particular to finite products of limit groups.  Indeed, these are virtually RFRS because they are virtually special \cite{HsuWise2015}.

\bibliographystyle{halpha}
\bibliography{refs.bib}

\end{document}

%% file: top.tex
\usepackage{amsmath} 
\usepackage{amssymb} 
\usepackage{amsthm} 
\usepackage{yhmath}
\usepackage[english]{babel} 
\usepackage[font=small,justification=centering]{caption} 
\usepackage[nodayofweek]{datetime}
\usepackage[shortlabels]{enumitem} 
\usepackage[T1]{fontenc} 
\usepackage[utf8]{inputenc} 
\usepackage{ifthen} 
\usepackage{mathtools} 
\usepackage[dvipsnames]{xcolor} 
\usepackage[pdftex, pdfencoding=auto,  colorlinks=true]{hyperref} 
    \hypersetup{urlcolor=RoyalBlue, linkcolor=RoyalBlue,  citecolor=black}
\usepackage{setspace} 
\usepackage{tikz-cd} 
\usepackage{xfrac} 
\usepackage[capitalize]{cleveref}
\usepackage{crossreftools} 
\usepackage{stmaryrd} 

\makeatletter
\def\@tocline#1#2#3#4#5#6#7{\relax
  \ifnum #1>\c@tocdepth 
  \else
    \par \addpenalty\@secpenalty\addvspace{#2}%
    \begingroup \hyphenpenalty\@M
    \@ifempty{#4}{%
      \@tempdima\csname r@tocindent\number#1\endcsname\relax
    }{%
      \@tempdima#4\relax
    }%
    \parindent\z@ \leftskip#3\relax \advance\leftskip\@tempdima\relax
    \rightskip\@pnumwidth plus4em \parfillskip-\@pnumwidth
    #5\leavevmode\hskip-\@tempdima
      \ifcase #1
       \or\or \hskip 1em \or \hskip 2em \else \hskip 3em \fi%
      #6\nobreak\relax
    \dotfill\hbox to\@pnumwidth{\@tocpagenum{#7}}\par
    \nobreak
    \endgroup
  \fi}
\makeatother

\newtheorem{thm}{Theorem}[section]
\newtheorem*{thm*}{Theorem}

\crefname{lem}{Lemma}{Lemmata}
\newtheorem{lem}[thm]{Lemma}
\newtheorem{corollary}[thm]{Corollary}
\newtheorem{cor}[thm]{Corollary}
\newtheorem{prop}[thm]{Proposition}

\newtheorem{thmx}{Theorem}

\theoremstyle{definition}
\newtheorem{defn}[thm]{Definition}
\newtheorem{dfn}[thm]{Definition}

\newtheorem{rmk}[thm]{Remark}

\newtheorem{example}[thm]{Example}

\DeclareMathOperator{\Ker}{\mathrm{Ker}}

\DeclareMathOperator{\Hom}{\mathrm{Hom}}

\DeclareMathOperator{\Tor}{\mathrm{Tor}}
\DeclareMathOperator{\tr}{\mathrm{tr}}
\DeclareMathOperator{\im}{\mathrm{Im}}

\newcommand{\calc}{{\mathcal{C}}}
\newcommand{\cald}{{\mathcal{D}}}

\newcommand{\calg}{{\mathcal{G}}}

\newcommand{\TAP}{\mathsf{TAP}}

\newcommand{\MC}{\mathrm{MC}}

\newcommand{\ZZhat}{\widehat{\mathbb{Z}}}

\newcommand*{\BS}{\mathrm{BS}}

\DeclareMathOperator{\rk}{\mathrm{rk}}

\newcommand*{\Frac}{\mathrm{Frac}}

\newcommand{\NN}{\mathbb{N}}
\newcommand{\ZZ}{\mathbb{Z}}

\newcommand{\RR}{\mathbb{R}}
\newcommand{\QQ}{\mathbb{Q}}
\newcommand{\FF}{\mathbb{F}}

\usepackage{tikz}
\usetikzlibrary{arrows,quotes}
\tikzstyle{blackNode}=[fill=black, draw=black, shape=circle]

%% file: acceptedVersion.bbl
\begin{thebibliography}{KLdGZ24}

\bibitem[Ago08]{Agol2008}
Ian Agol.
\newblock Criteria for virtual fibering.
\newblock {\em J. Topol.}, 1(2):269--284, 2008.
\newblock \href{https://dx.doi.org/10.1112/jtopol/jtn003}{{\ttfamily
  10.1112/jtopol/jtn003}}.

\bibitem[Ago13]{Agol2013}
Ian Agol.
\newblock The virtual {H}aken conjecture.
\newblock {\em Doc. Math.}, 18:1045--1087, 2013.
\newblock With an appendix by Agol, Daniel Groves, and Jason Manning.

\bibitem[BB97]{BestvinaBrady1997}
Mladen Bestvina and Noel Brady.
\newblock Morse theory and finiteness properties of groups.
\newblock {\em Invent. Math.}, 129(3):445--470, 1997.
\newblock \href{https://dx.doi.org/10.1007/s002220050168}{{\ttfamily
  10.1007/s002220050168}}.

\bibitem[BF09]{BestvinaFeighn2009}
Mladen Bestvina and Mark Feighn.
\newblock Notes on {S}ela's work: limit groups and {M}akanin-{R}azborov
  diagrams.
\newblock In {\em Geometric and cohomological methods in group theory}, volume
  358 of {\em London Math. Soc. Lecture Note Ser.}, pages 1--29. Cambridge
  Univ. Press, Cambridge, 2009.

\bibitem[BG04]{BridsonGrunewald2004}
Martin~R. Bridson and Fritz~J. Grunewald.
\newblock Grothendieck's problems concerning profinite completions and
  representations of groups.
\newblock {\em Ann. of Math. (2)}, 160(1):359--373, 2004.
\newblock \href{https://dx.doi.org/10.4007/annals.2004.160.359}{{\ttfamily
  10.4007/annals.2004.160.359}}.

\bibitem[BG10]{BieriGeoghegan2010}
Robert Bieri and Ross Geoghegan.
\newblock Sigma invariants of direct products of groups.
\newblock {\em Groups Geom. Dyn.}, 4(2):251--261, 2010.
\newblock \href{https://dx.doi.org/10.4171/GGD/82}{{\ttfamily 10.4171/GGD/82}}.

\bibitem[BHMS09]{BridsonHowieMillerShort2009}
Martin~R. Bridson, James Howie, Charles~F. Miller, III, and Hamish Short.
\newblock Subgroups of direct products of limit groups.
\newblock {\em Ann. of Math. (2)}, 170(3):1447--1467, 2009.
\newblock \href{https://dx.doi.org/10.4007/annals.2009.170.1447}{{\ttfamily
  10.4007/annals.2009.170.1447}}.

\bibitem[BL20]{BrownLeary2020}
Thomas Brown and Ian~J Leary.
\newblock Groups of type {$FP$} via graphical small cancellation, 2020.

\bibitem[BM22]{BattistaMartelli2022}
Ludovico Battista and Bruno Martelli.
\newblock Hyperbolic 4-manifolds with perfect circle-valued {M}orse functions.
\newblock {\em Trans. Amer. Math. Soc.}, 375(4):2597--2625, 2022.
\newblock \href{https://dx.doi.org/10.1090/tran/8542}{{\ttfamily
  10.1090/tran/8542}}.

\bibitem[BMRS20]{Bridsonetal2020}
M.~R. Bridson, D.~B. McReynolds, A.~W. Reid, and R.~Spitler.
\newblock Absolute profinite rigidity and hyperbolic geometry.
\newblock {\em Ann. of Math. (2)}, 192(3):679--719, 2020.
\newblock \href{https://dx.doi.org/10.4007/annals.2020.192.3.1}{{\ttfamily
  10.4007/annals.2020.192.3.1}}.

\bibitem[BNS87]{BieriNeumannStrebel1987}
Robert Bieri, Walter~D. Neumann, and Ralph Strebel.
\newblock A geometric invariant of discrete groups.
\newblock {\em Invent. Math.}, 90(3):451--477, 1987.
\newblock \href{https://dx.doi.org/10.1007/BF01389175}{{\ttfamily
  10.1007/BF01389175}}.

\bibitem[BR88]{BieriRenz1988}
Robert Bieri and Burkhardt Renz.
\newblock Valuations on free resolutions and higher geometric invariants of
  groups.
\newblock {\em Comment. Math. Helv.}, 63(3):464--497, 1988.
\newblock \href{https://dx.doi.org/10.1007/BF02566775}{{\ttfamily
  10.1007/BF02566775}}.

\bibitem[Bri16]{Bridson2016}
Martin~R. Bridson.
\newblock The strong profinite genus of a finitely presented group can be
  infinite.
\newblock {\em J. Eur. Math. Soc. (JEMS)}, 18(9):1909--1918, 2016.
\newblock \href{https://dx.doi.org/10.4171/JEMS/633}{{\ttfamily
  10.4171/JEMS/633}}.

\bibitem[Bro87]{Brown1987}
Kenneth~S. Brown.
\newblock Trees, valuations, and the {B}ieri-{N}eumann-{S}trebel invariant.
\newblock {\em Invent. Math.}, 90(3):479--504, 1987.

\bibitem[Bro94]{Brown1982}
Kenneth~S. Brown.
\newblock {\em Cohomology of groups}, volume~87 of {\em Graduate Texts in
  Mathematics}.
\newblock Springer-Verlag, New York, 1994.
\newblock Corrected reprint of the 1982 original.

\bibitem[BS78]{BieriStrebel1978}
Robert Bieri and Ralph Strebel.
\newblock Almost finitely presented soluble groups.
\newblock {\em Comment. Math. Helv.}, 53(2):258--278, 1978.
\newblock \href{https://dx.doi.org/10.1007/BF02566077}{{\ttfamily
  10.1007/BF02566077}}.

\bibitem[CR12]{ClayRolfson2012}
Adam Clay and Dale Rolfsen.
\newblock Ordered groups, eigenvalues, knots, surgery and {L}-spaces.
\newblock {\em Math. Proc. Camb. Philos. Soc.}, 152(1):115--129, 2012.
\newblock \href{https://dx.doi.org/10.1017/S0305004111000557}{{\ttfamily
  10.1017/S0305004111000557}}.

\bibitem[CRDK23]{Casals_EtAl23}
Montserrat Casals-Ruiz, Andrew Duncan, and Ilya Kazachkov.
\newblock Limit groups over coherent right-angled {A}rtin groups.
\newblock {\em Publ. Mat.}, 67(1):199--257, 2023.
\newblock \href{https://dx.doi.org/10.5565/publmat6712305}{{\ttfamily
  10.5565/publmat6712305}}.

\bibitem[CRK11]{Casals_Kazachkov2011}
Montserrat Casals-Ruiz and Ilya Kazachkov.
\newblock On systems of equations over free partially commutative groups.
\newblock {\em Mem. Amer. Math. Soc.}, 212(999):viii+153, 2011.
\newblock \href{https://dx.doi.org/10.1090/S0065-9266-2010-00628-8}{{\ttfamily
  10.1090/S0065-9266-2010-00628-8}}.

\bibitem[EM80]{EckmannMueller1980}
Beno Eckmann and Heinz M{\"{u}}ller.
\newblock Poincar\'{e} duality groups of dimension two.
\newblock {\em Comment. Math. Helv.}, 55(4):510--520, 1980.
\newblock \href{https://dx.doi.org/10.1007/BF02566702}{{\ttfamily
  10.1007/BF02566702}}.

\bibitem[Far72]{Farrell1972}
F.~T. Farrell.
\newblock The obstruction to fibering a manifold over a circle.
\newblock {\em Indiana University Mathematics Journal}, 21:315--346, 1971/72.
\newblock \href{https://dx.doi.org/10.1512/iumj.1971.21.21024}{{\ttfamily
  10.1512/iumj.1971.21.21024}}.

\bibitem[FHL24]{FisherHughesLeary2023}
S.P. Fisher, S.~Hughes, and I.J. Leary.
\newblock Homological growth of {A}rtin kernels in positive characteristic.
\newblock {\em Math. Ann.}, 389:819--843, 2024.
\newblock \href{https://dx.doi.org/10.1007/s00208-023-02663-1}{{\ttfamily
  10.1007/s00208-023-02663-1}}.

\bibitem[Fis23]{Fisher2022}
Sam~P. Fisher.
\newblock Algebraic fibring of a hyperbolic 7-manifold.
\newblock {\em Bull. Lond. Math. Soc.}, 55(3):1347--1357, 2023.
\newblock \href{https://dx.doi.org/10.1112/blms.12795}{{\ttfamily
  10.1112/blms.12795}}.

\bibitem[Fis24]{Fisher2024}
Sam~P. Fisher.
\newblock Improved algebraic fibrings.
\newblock {\em Compositio Mathematica}, 160(9):2203--2227, 2024.
\newblock \href{https://dx.doi.org/10.1112/S0010437X24007309}{{\ttfamily
  10.1112/S0010437X24007309}}.

\bibitem[FK06]{FriedlKim2006}
Stefan Friedl and Taehee Kim.
\newblock The {T}hurston norm, fibered manifolds and twisted {A}lexander
  polynomials.
\newblock {\em Topology}, 45(6):929--953, 2006.
\newblock \href{https://dx.doi.org/10.1016/j.top.2006.06.003}{{\ttfamily
  10.1016/j.top.2006.06.003}}.

\bibitem[FL17]{FriedlLueck2017}
Stefan Friedl and Wolfgang L\"{u}ck.
\newblock Universal {$L^2$}-torsion, polytopes and applications to 3-manifolds.
\newblock {\em Proc. Lond. Math. Soc. (3)}, 114(6):1114--1151, 2017.
\newblock \href{https://dx.doi.org/10.1112/plms.12035}{{\ttfamily
  10.1112/plms.12035}}.

\bibitem[Fru23]{Fruchter23}
Jonathan Fruchter.
\newblock Limit groups over coherent right-angled {A}rtin groups are cyclic
  subgroup separable.
\newblock {\em Michigan Math. J.}, 73(5):909--923, 2023.
\newblock \href{https://dx.doi.org/10.1307/mmj/20216031}{{\ttfamily
  10.1307/mmj/20216031}}.

\bibitem[FT20]{FriedlTillmann2020}
Stefan Friedl and Stephan Tillmann.
\newblock Two-generator one-relator groups and marked polytopes.
\newblock {\em Ann. Inst. Fourier (Grenoble)}, 70(2):831--879, 2020.

\bibitem[FV08a]{FriedlVidussi2008}
Stefan Friedl and Stefano Vidussi.
\newblock Symplectic {$S^1\times N^3$}, subgroup separability, and vanishing
  {T}hurston norm.
\newblock {\em J. Amer. Math. Soc.}, 21(2):597--610, 2008.
\newblock \href{https://dx.doi.org/10.1090/S0894-0347-07-00577-2}{{\ttfamily
  10.1090/S0894-0347-07-00577-2}}.

\bibitem[FV08b]{FriedlVidussi2008a}
Stefan Friedl and Stefano Vidussi.
\newblock Twisted {A}lexander polynomials and symplectic structures.
\newblock {\em Amer. J. Math.}, 130(2):455--484, 2008.
\newblock \href{https://dx.doi.org/10.1353/ajm.2008.0014}{{\ttfamily
  10.1353/ajm.2008.0014}}.

\bibitem[FV11a]{FriedlVidussi2011survey}
Stefan Friedl and Stefano Vidussi.
\newblock A survey of twisted {A}lexander polynomials.
\newblock In {\em The mathematics of knots}, volume~1 of {\em Contrib. Math.
  Comput. Sci.}, pages 45--94. Springer, Heidelberg, 2011.
\newblock \href{https://dx.doi.org/10.1007/978-3-642-15637-3\_3}{{\ttfamily
  10.1007/978-3-642-15637-3\_3}}.

\bibitem[FV11b]{FriedlVidussi2011annals}
Stefan Friedl and Stefano Vidussi.
\newblock Twisted {A}lexander polynomials detect fibered 3-manifolds.
\newblock {\em Ann. of Math. (2)}, 173(3):1587--1643, 2011.
\newblock \href{https://dx.doi.org/10.4007/annals.2011.173.3.8}{{\ttfamily
  10.4007/annals.2011.173.3.8}}.

\bibitem[FV13]{FriedlVidussi2013}
Stefan Friedl and Stefano Vidussi.
\newblock A vanishing theorem for twisted {A}lexander polynomials with
  applications to symplectic 4-manifolds.
\newblock {\em J. Eur. Math. Soc. (JEMS)}, 15(6):2027--2041, 2013.
\newblock \href{https://dx.doi.org/10.4171/JEMS/412}{{\ttfamily
  10.4171/JEMS/412}}.

\bibitem[FV16]{FriedlVidussi2016}
Stefan Friedl and Stefano Vidussi.
\newblock Rank gradients of infinite cyclic covers of {K}\"ahler manifolds.
\newblock {\em J. Group Theory}, 19(5):941--957, 2016.
\newblock \href{https://dx.doi.org/10.1515/jgth-2016-0019}{{\ttfamily
  10.1515/jgth-2016-0019}}.

\bibitem[Geh98]{Gerhke1998}
Ralf Gehrke.
\newblock The higher geometric invariants for groups with sufficient
  commutativity.
\newblock {\em Commun. Algebra}, 26(4):1097--1115, 1998.
\newblock \href{https://dx.doi.org/10.1080/00927879808826186}{{\ttfamily
  10.1080/00927879808826186}}.

\bibitem[GJZZ08]{Grunewaldetal2008}
F.~Grunewald, A.~Jaikin-Zapirain, and P.~A. Zalesskii.
\newblock Cohomological goodness and the profinite completion of {B}ianchi
  groups.
\newblock {\em Duke Math. J.}, 144(1):53--72, 2008.
\newblock \href{https://dx.doi.org/10.1215/00127094-2008-031}{{\ttfamily
  10.1215/00127094-2008-031}}.

\bibitem[GKM05]{GodaKitanoMorifuji2005}
Hiroshi Goda, Teruaki Kitano, and Takayuki Morifuji.
\newblock Reidemeister torsion, twisted {Alexander} polynomial and fibered
  knots.
\newblock {\em Comment. Math. Helv.}, 80(1):51--61, 2005.
\newblock \href{https://dx.doi.org/10.4171/CMH/3}{{\ttfamily 10.4171/CMH/3}}.

\bibitem[GPS80]{GrunewaldPickelSegal1980}
F.~J. Grunewald, P.~F. Pickel, and D.~Segal.
\newblock Polycyclic groups with isomorphic finite quotients.
\newblock {\em Ann. of Math. (2)}, 111(1):155--195, 1980.
\newblock \href{https://dx.doi.org/10.2307/1971220}{{\ttfamily
  10.2307/1971220}}.

\bibitem[Hil02]{Hillman2002}
J.~A. Hillman.
\newblock {\em Four-manifolds, geometries and knots}, volume~5 of {\em Geometry
  \& Topology Monographs}.
\newblock Geometry \& Topology Publications, Coventry, 2002.

\bibitem[Hil20a]{Hillman2020b}
Jonathan~A. Hillman.
\newblock ${PD}_3$-groups and {HNN} extensions, 2020.
\newblock \href{https://arxiv.org/abs/2004.03803}{{\ttfamily arXiv:2004.03803
  [math.GR]}}.

\bibitem[Hil20b]{Hillman2020}
Jonathan~A. Hillman.
\newblock {\em Poincar\'{e} duality in dimension 3}, volume~3 of {\em The Open
  Book Series}.
\newblock Mathematical Sciences Publishers, Berkeley, CA, 2020.

\bibitem[HK21]{HenselKielak2021}
Sebastian Hensel and Dawid Kielak.
\newblock Handlebody bundles and polytopes.
\newblock {\em Algebr. Geom. Topol.}, 21(7):3445--3458, 2021.
\newblock \href{https://dx.doi.org/10.2140/agt.2021.21.3445}{{\ttfamily
  10.2140/agt.2021.21.3445}}.

\bibitem[Hug70]{Hughes1970}
Ian Hughes.
\newblock Division rings of fractions for group rings.
\newblock {\em Comm. Pure Appl. Math.}, 23:181--188, 1970.
\newblock \href{https://dx.doi.org/10.1002/cpa.3160230205}{{\ttfamily
  10.1002/cpa.3160230205}}.

\bibitem[Hug22]{Hughes2022}
Sam Hughes.
\newblock Irreducible lattices fibring over the circle, 2022.
\newblock \href{https://arxiv.org/abs/2201.06525}{{\ttfamily arXiv:2201.06525
  [math.GR]}}.

\bibitem[HW15]{HsuWise2015}
Tim Hsu and Daniel~T. Wise.
\newblock Cubulating malnormal amalgams.
\newblock {\em Invent. Math.}, 199(2):293--331, 2015.
\newblock \href{https://dx.doi.org/10.1007/s00222-014-0513-4}{{\ttfamily
  10.1007/s00222-014-0513-4}}.

\bibitem[IMM23]{ItalianoMartelliMigliorini2021}
Giovanni Italiano, Bruno Martelli, and Matteo Migliorini.
\newblock Hyperbolic 5-manifolds that fiber over {$S^1$}.
\newblock {\em Invent. Math.}, 231(1):1--38, 2023.
\newblock \href{https://dx.doi.org/10.1007/s00222-022-01141-w}{{\ttfamily
  10.1007/s00222-022-01141-w}}.

\bibitem[IMM24]{ItalianoMartelliMigliorini2020}
Giovanni Italiano, Bruno Martelli, and Matteo Migliorini.
\newblock Hyperbolic manifolds that fibre algebraically up to dimension 8.
\newblock {\em J. Inst. Math. Jussieu}, 23(2):609--646, 2024.
\newblock \href{https://dx.doi.org/10.1017/S1474748022000536}{{\ttfamily
  10.1017/S1474748022000536}}.

\bibitem[IMP24]{IsenrichMartelliPy2022}
Claudio~Llosa Isenrich, Bruno Martelli, and Pierre Py.
\newblock Hyperbolic groups containing subgroups of type {$\mathcal{F}_3$} not
  {$\mathcal{F}_4$}.
\newblock {\em J. Differential Geom.}, 127(3):1121--1147, 2024.
\newblock \href{https://dx.doi.org/10.4310/jdg/1721071498}{{\ttfamily
  10.4310/jdg/1721071498}}.

\bibitem[JNW21]{JankiewiczNorinWise2021}
Kasia Jankiewicz, Sergey Norin, and Daniel~T. Wise.
\newblock Virtually fibering right-angled {C}oxeter groups.
\newblock {\em J. Inst. Math. Jussieu}, 20(3):957--987, 2021.
\newblock \href{https://dx.doi.org/10.1017/S1474748019000422}{{\ttfamily
  10.1017/S1474748019000422}}.

\bibitem[JW72]{JohnsonWall1972}
F.~E.~A. Johnson and C.~T.~C. Wall.
\newblock On groups satisfying {P}oincar\'e duality.
\newblock {\em Ann. of Math. (2)}, 96:592--598, 1972.

\bibitem[JZ20]{JaikinZapirain2020}
Andrei Jaikin-Zapirain.
\newblock Recognition of being fibered for compact 3{\textendash}manifolds.
\newblock {\em Geom. Topol.}, 24(1):409--420, mar 2020.
\newblock \href{https://dx.doi.org/10.2140/gt.2020.24.409}{{\ttfamily
  10.2140/gt.2020.24.409}}.

\bibitem[JZ21]{Jaikin2021}
Andrei Jaikin-Zapirain.
\newblock The universality of {H}ughes-free division rings.
\newblock {\em Selecta Math. (N.S.)}, 27(4):Paper No. 74, 33, 2021.
\newblock \href{https://dx.doi.org/10.1007/s00029-021-00691-w}{{\ttfamily
  10.1007/s00029-021-00691-w}}.

\bibitem[Kie20a]{Kielak2020BNS}
Dawid Kielak.
\newblock The {B}ieri-{N}eumann-{S}trebel invariants via {N}ewton polytopes.
\newblock {\em Invent. Math.}, 219(3):1009--1068, 2020.
\newblock \href{https://dx.doi.org/10.1007/s00222-019-00919-9}{{\ttfamily
  10.1007/s00222-019-00919-9}}.

\bibitem[Kie20b]{Kielak2020RFRS}
Dawid Kielak.
\newblock Residually finite rationally solvable groups and virtual fibring.
\newblock {\em J. Amer. Math. Soc.}, 33(2):451--486, 2020.
\newblock \href{https://dx.doi.org/10.1090/jams/936}{{\ttfamily
  10.1090/jams/936}}.

\bibitem[KK21]{KielakKropholler2021}
Dawid Kielak and Peter Kropholler.
\newblock Isoperimetric inequalities for {P}oincar\'{e} duality groups.
\newblock {\em Proc. Amer. Math. Soc.}, 149(11):4685--4698, 2021.
\newblock \href{https://dx.doi.org/10.1090/proc/15596}{{\ttfamily
  10.1090/proc/15596}}.

\bibitem[KL99]{KirkLivingston1999}
Paul Kirk and Charles Livingston.
\newblock Twisted {A}lexander invariants, {R}eidemeister torsion, and
  {C}asson-{G}ordon invariants.
\newblock {\em Topology}, 38(3):635--661, 1999.
\newblock \href{https://dx.doi.org/10.1016/S0040-9383(98)00039-1}{{\ttfamily
  10.1016/S0040-9383(98)00039-1}}.

\bibitem[KLdGZ24]{KochloukovaLopezGZ24}
Dessislava~H. Kochloukova and Jone Lopez~de Gamiz~Zearra.
\newblock On subdirect products of type {$FP_n$} of limit groups over {D}roms
  {RAAG}s.
\newblock {\em Math. Proc. Cambridge Philos. Soc.}, 176(2):417--440, 2024.
\newblock \href{https://dx.doi.org/10.1017/s0305004123000579}{{\ttfamily
  10.1017/s0305004123000579}}.

\bibitem[KLS20]{KrophollerLearySoroko2020}
Robert~P. Kropholler, Ian~J. Leary, and Ignat Soroko.
\newblock Uncountably many quasi-isometry classes of groups of type {$FP$}.
\newblock {\em Amer. J. Math.}, 142(6):1931--1944, 2020.
\newblock \href{https://dx.doi.org/10.1353/ajm.2020.0048}{{\ttfamily
  10.1353/ajm.2020.0048}}.

\bibitem[Kro18]{Kropholler2018a}
Robert Kropholler.
\newblock Almost hyperbolic groups with almost finitely presented subgroups,
  2018.
\newblock \href{https://arxiv.org/abs/1802.01658}{{\ttfamily arXiv:1802.01658
  [math.GR]}}.

\bibitem[KZ08]{KochloukovaZalesskii2008}
Dessislava~H. Kochloukova and Pavel~A. Zalesskii.
\newblock Profinite and pro-{$p$} completions of {P}oincar\'{e} duality groups
  of dimension 3.
\newblock {\em Trans. Amer. Math. Soc.}, 360(4):1927--1949, 2008.
\newblock \href{https://dx.doi.org/10.1090/S0002-9947-07-04519-9}{{\ttfamily
  10.1090/S0002-9947-07-04519-9}}.

\bibitem[LdGZ22]{LopezGZ22}
Jone Lopez~de Gamiz~Zearra.
\newblock Subgroups of direct products of limit groups over {D}roms {RAAG}s.
\newblock {\em Algebr. Geom. Topol.}, 22(7):3485--3510, 2022.
\newblock \href{https://dx.doi.org/10.2140/agt.2022.22.3485}{{\ttfamily
  10.2140/agt.2022.22.3485}}.

\bibitem[Lea18a]{Leary2018b}
Ian~J. Leary.
\newblock Subgroups of almost finitely presented groups.
\newblock {\em Math. Ann.}, 372(3-4):1383--1391, 2018.
\newblock \href{https://dx.doi.org/10.1007/s00208-018-1689-5}{{\ttfamily
  10.1007/s00208-018-1689-5}}.

\bibitem[Lea18b]{Leary2018}
Ian~J. Leary.
\newblock Uncountably many groups of type {$FP$}.
\newblock {\em Proc. Lond. Math. Soc. (3)}, 117(2):246--276, 2018.
\newblock \href{https://dx.doi.org/10.1112/plms.12135}{{\ttfamily
  10.1112/plms.12135}}.

\bibitem[LIP24]{IsenrichPy2022}
Claudio Llosa~Isenrich and Pierre Py.
\newblock Subgroups of hyperbolic groups, finiteness properties and complex
  hyperbolic lattices.
\newblock {\em Invent. Math.}, 235(1):233--254, 2024.
\newblock \href{https://dx.doi.org/10.1007/s00222-023-01223-3}{{\ttfamily
  10.1007/s00222-023-01223-3}}.

\bibitem[Liu23]{Liu2020}
Yi~Liu.
\newblock Finite-volume hyperbolic 3-manifolds are almost determined by their
  finite quotient groups.
\newblock {\em Invent. Math.}, 231(2):741--804, 2023.
\newblock \href{https://dx.doi.org/10.1007/s00222-022-01155-4}{{\ttfamily
  10.1007/s00222-022-01155-4}}.

\bibitem[Lor08]{Lorensen2008}
Karl Lorensen.
\newblock Groups with the same cohomology as their profinite completions.
\newblock {\em J. Algebra}, 320(4):1704--1722, 2008.
\newblock \href{https://dx.doi.org/10.1016/j.jalgebra.2008.03.013}{{\ttfamily
  10.1016/j.jalgebra.2008.03.013}}.

\bibitem[Mil68]{Milnor1967}
John~W. Milnor.
\newblock Infinite cyclic coverings.
\newblock In {\em Conference on the {T}opology of {M}anifolds ({M}ichigan
  {S}tate {U}niv., {E}. {L}ansing, {M}ich., 1967)}, pages 115--133. Prindle,
  Weber \& Schmidt, Boston, Mass., 1968.

\bibitem[NS07a]{NikolovSegal2007I}
Nikolay Nikolov and Dan Segal.
\newblock On finitely generated profinite groups. {I}. {S}trong completeness
  and uniform bounds.
\newblock {\em Ann. of Math. (2)}, 165(1):171--238, 2007.
\newblock \href{https://dx.doi.org/10.4007/annals.2007.165.171}{{\ttfamily
  10.4007/annals.2007.165.171}}.

\bibitem[NS07b]{NikolovSegal2007II}
Nikolay Nikolov and Dan Segal.
\newblock On finitely generated profinite groups. {II}. {P}roducts in
  quasisimple groups.
\newblock {\em Ann. of Math. (2)}, 165(1):239--273, 2007.
\newblock \href{https://dx.doi.org/10.4007/annals.2007.165.239}{{\ttfamily
  10.4007/annals.2007.165.239}}.

\bibitem[Pic71]{Pickel1971}
P.~F. Pickel.
\newblock Finitely generated nilpotent groups with isomorphic finite quotients.
\newblock {\em Trans. Amer. Math. Soc.}, 160:327--341, 1971.
\newblock \href{https://dx.doi.org/10.2307/1995809}{{\ttfamily
  10.2307/1995809}}.

\bibitem[Pic74]{Pickel1974}
P.~F. Pickel.
\newblock Metabelian groups with the same finite quotients.
\newblock {\em Bull. Austral. Math. Soc.}, 11:115--120, 1974.
\newblock \href{https://dx.doi.org/10.1017/S0004972700043689}{{\ttfamily
  10.1017/S0004972700043689}}.

\bibitem[PT86]{PlatonovTavgen1986}
V.~P. Platonov and O.~I. Tavgen'.
\newblock On the {G}rothendieck problem of profinite completions of groups.
\newblock {\em Dokl. Akad. Nauk SSSR}, 288(5):1054--1058, 1986.

\bibitem[Pyb04]{Pyber2004}
L\'{a}szl\'{o} Pyber.
\newblock Groups of intermediate subgroup growth and a problem of
  {G}rothendieck.
\newblock {\em Duke Math. J.}, 121(1):169--188, 2004.
\newblock \href{https://dx.doi.org/10.1215/S0012-7094-04-12115-3}{{\ttfamily
  10.1215/S0012-7094-04-12115-3}}.

\bibitem[Rot09]{Rotman2009}
Joseph~J. Rotman.
\newblock {\em An introduction to homological algebra}.
\newblock Universitext. Springer, New York, second edition, 2009.
\newblock \href{https://dx.doi.org/10.1007/b98977}{{\ttfamily 10.1007/b98977}}.

\bibitem[RZ10]{RibesZalesskii2010}
Luis Ribes and Pavel Zalesskii.
\newblock {\em Profinite Groups}.
\newblock Springer-Verlag GmbH, March 2010.

\bibitem[Sch08]{Schutz2008}
Dirk Sch\"{u}tz.
\newblock On the direct product conjecture for sigma invariants.
\newblock {\em Bull. Lond. Math. Soc.}, 40(4):675--684, 2008.
\newblock \href{https://dx.doi.org/10.1112/blms/bdn048}{{\ttfamily
  10.1112/blms/bdn048}}.

\bibitem[Sco73]{Scott1973a}
G.~P. Scott.
\newblock Compact submanifolds of 3-manifolds.
\newblock {\em Journal of the London Mathematical Society}, s2-7(2):246--250,
  nov 1973.
\newblock \href{https://dx.doi.org/10.1112/jlms/s2-7.2.246}{{\ttfamily
  10.1112/jlms/s2-7.2.246}}.

\bibitem[Sel01]{Sela2001I}
Zlil Sela.
\newblock Diophantine geometry over groups. {I}. {M}akanin-{R}azborov diagrams.
\newblock {\em Publ. Math. Inst. Hautes \'Etudes Sci.}, (93):31--105, 2001.
\newblock \href{https://dx.doi.org/10.1007/s10240-001-8188-y}{{\ttfamily
  10.1007/s10240-001-8188-y}}.

\bibitem[Str77]{Strebel1977}
R.~Strebel.
\newblock A remark on subgroups of infinite index in {P}oincar\'{e} duality
  groups.
\newblock {\em Comment. Math. Helv.}, 52(3):317--324, 1977.
\newblock \href{https://dx.doi.org/10.1007/BF02567371}{{\ttfamily
  10.1007/BF02567371}}.

\bibitem[Tur01]{Turaev2001}
Vladimir Turaev.
\newblock {\em Introduction to combinatorial torsions}.
\newblock Lectures in Mathematics ETH Z\"{u}rich. Birkh\"{a}user Verlag, Basel,
  2001.
\newblock \href{https://dx.doi.org/10.1007/978-3-0348-8321-4}{{\ttfamily
  10.1007/978-3-0348-8321-4}}, Notes taken by Felix Schlenk.

\bibitem[Wei94]{Weibel1994}
Charles~A. Weibel.
\newblock {\em An introduction to homological algebra}, volume~38 of {\em
  Cambridge Studies in Advanced Mathematics}.
\newblock Cambridge University Press, Cambridge, 1994.
\newblock \href{https://dx.doi.org/10.1017/CBO9781139644136}{{\ttfamily
  10.1017/CBO9781139644136}}.

\bibitem[Wil08]{Wilton2008}
Henry Wilton.
\newblock Hall's theorem for limit groups.
\newblock {\em Geometric and Functional Analysis}, 18(1):271--303, mar 2008.
\newblock \href{https://dx.doi.org/10.1007/s00039-008-0657-8}{{\ttfamily
  10.1007/s00039-008-0657-8}}.

\bibitem[Wis12]{Wise2012}
Daniel~T. Wise.
\newblock {\em From riches to raags: 3-manifolds, right-angled {A}rtin groups,
  and cubical geometry}, volume 117 of {\em CBMS Regional Conference Series in
  Mathematics}.
\newblock Published for the Conference Board of the Mathematical Sciences,
  Washington, DC; by the American Mathematical Society, Providence, RI, 2012.
\newblock \href{https://dx.doi.org/10.1090/cbms/117}{{\ttfamily
  10.1090/cbms/117}}.

\bibitem[WZ10]{WiltonZalesskii2010}
Henry Wilton and Pavel Zalesskii.
\newblock Profinite properties of graph manifolds.
\newblock {\em Geom. Dedicata}, 147:29--45, 2010.
\newblock \href{https://dx.doi.org/10.1007/s10711-009-9437-3}{{\ttfamily
  10.1007/s10711-009-9437-3}}.

\bibitem[WZ17]{WiltonZalesskii2017}
Henry Wilton and Pavel Zalesskii.
\newblock Distinguishing geometries using finite quotients.
\newblock {\em Geometry {\&} Topology}, 21(1):345--384, feb 2017.
\newblock \href{https://dx.doi.org/10.2140/gt.2017.21.345}{{\ttfamily
  10.2140/gt.2017.21.345}}.

\end{thebibliography}
